\documentclass[reqno]{amsart}
\usepackage{yhmath,amsmath,amsfonts,amssymb,enumerate,amsthm,appendix,caption,lipsum,xparse}
\usepackage{enumitem}
\usepackage{mathrsfs}
\usepackage[english]{babel}
\usepackage{cmap,mathtools}
\usepackage{hyperref}
\usepackage[utf8]{inputenc}
\usepackage{graphics} 
\usepackage{epsfig} 
\usepackage{graphicx}\usepackage{epstopdf}
\usepackage[pagewise]{lineno}

\usepackage{a4wide}

\usepackage[margin=0.95in]{geometry}
\usepackage{cite}
\usepackage[utf8]{inputenc}
\usepackage{xcolor}
\usepackage{hyperref}
\usepackage[hyperpageref]{backref}
\hypersetup{
    colorlinks=true,
    linkcolor=myblue,
    filecolor=magenta,      
    urlcolor=mygreen,
    citecolor=mygreen,
}
\definecolor{mygreen}{rgb}{0.01,0.6,0.2}
\definecolor{myblue}{rgb}{0.01, 0.18, 1.0}
\newtheorem{theorem}{Theorem}

\newtheorem{lemma}[theorem]{Lemma}
\newtheorem{corollary}[theorem]{Corollary}
\theoremstyle{definition}
\newtheorem{definition}[theorem]{Definition}
\newtheorem{remark}[theorem]{Remark}

\numberwithin{equation}{section}
\numberwithin{theorem}{section}
\numberwithin{equation}{section}
\numberwithin{theorem}{section}

\title[Mixed Schr\"{o}dinger system in a plane]{Existence of normalized ground state solution to a mixed Schr\"odinger system in a plane}
\author[A. Dixit, A. Esfahani, H. Hajaiej \& T. Mukherjee]{ASHUTOSH DIXIT$^{1}$, AMIN ESFAHANI$^2$, HICHEM HAJAIEJ$^{3}$ and  TUHINA MUKHERJEE$^1$}

\begin{document}
\def\dxy{\mathrm{d}x \mathrm{d}y}
\def\d{\mathrm{d}}
\def\dxz{\mathrm{d}x \mathrm{d}z}
\def\dyz{\mathrm{d}y \mathrm{d}z}
\def\R{\mathbb{R}}
\def\({\bigg(}
\def\){\bigg)}
\def\H{\mathcal{H}^{1,s}(\R^2)}
\def\Hi{\mathcal{H}^{1,s_i}(\R^2)}
\def\Hone{\mathcal{H}^{1,s}(\R^2)}
\def\Htwo{\mathcal{H}^{1,s}(\R^2)}
\def\ui{\int_{\R^2} \big(|u|^2 + |\partial_x u|^2 + |(-\Delta)^{s_{i}/2}_y u|^2 \big) \dxy}
\def\uone{|u|^2 + |\partial_x u|^2 + |(-\Delta)^{s/2}_{y} u|^2 }
\def\unone{|u_n|^2 + |\partial_x u_n|^2 + |(-\Delta)^{s/2}_{y} u_n|^2 }
\def\v2{\int_{\R^2} \big(|v|^2 + |\partial_x v|^2 + |(-\Delta)^{s/2}_yv|^2 \big) \dxy}
\def\uq{\(\int_{\R^2}|u|^q \dxy\)^{2/q}}
\def\u2si{\(\int_{\R^2}|u|^{2_{s_i}} \dxy\)^{2/2_{s_i}}}
\def\Dely{(-\Delta)_{y}^{s/2}}
\def\Delone{(-\Delta)_{y}^{s/2}}
\def\Deltwo{(-\Delta)_{y}^{s/2}}
\def\C_c{C_{c}^{\infty}(\R^2)}
\definecolor{dogwoodrose}{rgb}{0.84, 0.09, 0.41}
\definecolor{lava}{rgb}{0.81, 0.06, 0.13}
\def\tb{\color{blue}}
\definecolor{ao(english)}{rgb}{0.0, 0.5, 0.0}
\def\ao{\color{ao(english)}}
\definecolor{robineggblue}{rgb}{0.0, 0.8, 0.8}
\def\tro{\color{tiffanyblue}}
\definecolor{tiffanyblue}{rgb}{0.04, 0.73, 0.71}
\definecolor{tenné(tawny)}{rgb}{0.8, 0.34, 0.0}
\def\ten{\color{tenné(tawny)}}
\definecolor{aqua}{rgb}{0.0, 1.0, 1.0}
\def\aqua{\color{aqua}}
\maketitle 
\centerline{$^{1}$Department of Mathematics, Indian Institute of Technology Jodhpur,}
 \centerline{Rajasthan, 342030, India}
 \centerline{$^{2}$Department of Mathematics, Nazarbayev University, }
 \centerline{ Astana, 010000, Kazakhstan}
 \centerline{$^{3}$Department of Mathematics, California State University at Los Angeles,}
\centerline{Los Angeles, CA 90032,
USA}

 \begin{abstract}
In this paper, we establish the existence of positive ground state solutions for a class of mixed Schr\"{o}dinger systems with concave-convex nonlinearities in $\mathbb{R}^2$, subject to $L^2$-norm constraints; that is,
\[
\left\{
\begin{aligned}
	-\partial_{xx} u + (-\Delta)_y^s u + \lambda_1 u &= \mu_1 u^{p-1} + \beta r_1 u^{r_1-1} v^{r_2}, &&   \\
	-\partial_{xx} v + (-\Delta)_y^s v + \lambda_2 v &= \mu_2 v^{q-1} + \beta r_2 u^{r_1} v^{r_2-1}, &&  
\end{aligned}
\right.
\]
subject to the $L^2$-norm constraints:
\[
\int_{\mathbb{R}^2} u^2 \,\mathrm{d}x\mathrm{d}y = a \quad \text{and} \quad \int_{\mathbb{R}^2} v^2 \,\mathrm{d}x\mathrm{d}y = b,
\]
where $(x,y)\in \mathbb{R}^2$, $u, v  \geq 0$, $s \in \left(1/2, 1 \right)$, $\mu_1, \mu_2, \beta > 0$, $r_1, r_2 > 1$, the prescribed masses $a, b > 0$, and the parameters $\lambda_1, \lambda_2$ appear as Lagrange multipliers. Moreover, the exponents $p, q, r_1 + r_2$ satisfy:
\[
\frac{2(1+3s)}{1+s} < p, q, r_1 + r_2 < 2_s,
\]
where $2_s = \frac{2(1+s)}{1-s}$.  
To obtain our main existence results, we employ variational techniques such as the Mountain Pass Theorem, the Pohozaev manifold, Steiner rearrangement, and others, consolidating the works of Louis Jeanjean et al. \cite{jeanjean2024normalized}.

 \end{abstract} 
 \maketitle

\section{Introduction}

 In recent years, the study of nonlinear fractional Schrödinger equations has experienced significant growth due to their applications in nonlinear optics, ecology, biology, and fractional quantum mechanics. A particular focus has been on the mixed Schrödinger operator $\mathcal{L} = -\partial_{xx} + (-\Delta)_{y}^{s} + I$, which has been extensively explored in various contexts, as highlighted by   Esfahani et al.~\cite{MR3857062}.
They investigated the problem:
\begin{equation}\label{1.1}
	w + (-\Delta)_{x}^{s}w + (-\Delta)_{y}w = f(w), \quad (x,y) \in \mathbb{R}^{n}\times\mathbb{R}^{m},
\end{equation}
where $s \in (0,1)$ and $m, n \in \mathbb{N}$. The operator $(- \Delta)_{x}^{s}$ represents the fractional Laplacian in the $x$-variable, defined (up to a normalizing constant $C_{n,s}$) as
\[
(-\Delta)_{x}^{s}w(x,y) = C_{n,s}~ \text{P.V.} \int_{\mathbb{R}^n} \frac{w(x,y) - w(z,y)}{|x-z|^{n+2s}} \, \mathrm{d}z.
\]
Here, the nonlinearity $f \in C^1(\mathbb{R}, \mathbb{R})$ satisfies the subcritical growth assumption, without the Ambrosetti-Rabinowitz condition. The operator $\mathcal{L}$ is particularly important for modeling anisotropic diffusion in stochastic processes, such as Brownian motion and Lévy–Itô processes, where directional sensitivity is crucial. These systems exhibit diffusion behavior that varies with direction, accurately capturing scenarios where movement is non-uniform in all directions. Esfahani et al.~established that equation \eqref{1.1} possesses a positive ground state solution, which is axially symmetric. This operator has also been studied in simpler models~\cite{MR2911421,MR2963799}.
Felmer and Wang~\cite{MR3918177} further explored similar problems, focusing on the qualitative properties of positive solutions under various assumptions on $f$. More recently, Gou et al.~\cite{gou2023solitary} provided a comprehensive study on the existence of solutions to a specific problem on a plane:
\begin{equation}\label{1.2}
	-\partial_{xx} w + (-\Delta)_{y}^{s} w + w = w^{p-1} \quad \text{in} ~ \mathbb{R}^2.
\end{equation}
They demonstrated that equation \eqref{1.2} admits a positive ground state solution that is axially symmetric when $2 < p < 2_s$. Furthermore, by applying the Pohozaev identity, they showed that the problem admits only a trivial solution for $p \geq 2_s$.\\
In 2024, Esfahani   et al.~\cite{esfahani2024new} discussed the anisotropic nonlinear Schrödinger equation on the plane, investigating the equation:
\begin{equation}\label{1.3}
	-\partial_{xx} w + (-\Delta)_{y}^{s} w + \lambda w = |w|^{p-2}w \quad \text{in} ~ \mathbb{R}^2.
\end{equation}
They established the existence of normalized solutions to equation \eqref{1.3} in subcritical, critical, and supercritical cases, utilizing the Pohozaev identity and the maximum principle. To address the challenge of finding normalized solutions for nonlinear Schrödinger equations, particularly under the $L^2$-norm constraint
\[
\int_{\mathbb{R}^2} |w|^2 \;\mathrm{d}x \mathrm{d}y = c > 0,
\]
we begin by exploring solutions to equation \eqref{eq:20220902-maine1}. The study of such solutions has gained significant attention in the mathematical community, with pioneering work by Lions \cite{MR834360} marking a key development through the concentration-compactness method. This approach provided a foundational framework for proving the existence of normalized solutions in subcritical cases, which has since been expanded by many researchers. For example, \cite{MR2100909} extended Lions' method, covering a wider range of nonlinearities. However, the problem becomes notably more complex in the critical and supercritical cases, where Bellazzini and Jeanjean \cite{MR3510005} introduced an innovative constraint to advance the field. Further advancements have been made by Hajaiej and Song \cite{hajaiej2023number,hajaiej2022general,hajaiej2023strict}, who developed comprehensive techniques to address the existence, non-existence, multiplicity, and uniqueness of normalized solutions.
Despite these advancements, the anisotropic case presents unique challenges that have not been thoroughly investigated. Specifically, crucial aspects such as the maximum principle and Pohozaev identity for solutions to equation \eqref{eq:20220902-maine1} remain largely unexplored. These aspects are critical for understanding the behavior of the anisotropic operator, requiring novel approaches distinct from those used in the fractional setting, which, while technically demanding, did not necessitate fundamentally new ideas. Our research aims to address these gaps by investigating these fundamental properties in the anisotropic context. This work is essential to advancing the understanding of normalized solutions within this more complex framework. For related studies on existence, non-existence, and the orbital stability of standing waves under optimal nonlinearities, the reader may refer to \cite{MR3070759,MR3311080,MR2883850,hajaiej2023existence}. For foundational studies on Schrödinger systems with local and nonlocal nonlinearities, as well as fractional Schrödinger systems with general nonlinearities, the reader may refer to \cite{hajaiej2023normalized,hajaiej2010}

In this paper, our aim is to establish the existence of normalized solutions for the following anisotropic nonlinear Schrödinger system on the plane:
\begin{equation}\label{eq:20220902-maine1}
	\begin{cases}
		-\partial_{xx}u + (-\Delta)_{y}^{s} u+\lambda_1u = \mu_1 |u|^{p-2}u+\beta r_1|u|^{r_1-2}u|v|^{r_2}, \quad &  (x,y)\in\mathbb{R}^2,\\
		-\partial_{xx}v + (-\Delta)_{y}^{s} v+\lambda_2v = \mu_2 |v|^{q-2}v+\beta r_2|u|^{r_1}|v|^{r_2-2}v,  \\
		\|u\|_{L^2(\mathbb{R}^2)}^2 = a, \, \|v\|_{L^2(\mathbb{R}^2)}^2 = b,
	\end{cases}
\end{equation}
where $s \in \left( \frac{1}{2},1 \right)$, $\mu_1, \mu_2 > 0$, $\beta > 0$, $r_1, r_2 > 1$, and the prescribed masses $a, b > 0$. Moreover, we assume:
\[
\frac{2(1+3s)}{1+s} < p, q, r_1 + r_2 < 2_s,
\]
where $2_s = \frac{2(1+s)}{1-s}$. 

We first define the fractional Sobolev-Liouville space $\mathcal{H}^{1,s}(\mathbb{R}^2)$. Let   $s \in (0,1)$. The fractional Sobolev-Liouville space $\mathcal{H}^{1,s}$ is the set of all functions $u \in L^2(\mathbb{R}^2)$ such that:
\[
\|u\|_{\mathcal{H}^{1,s}(\mathbb{R}^2)} = \|u\|_{L^2(\mathbb{R}^2)} + \|\partial_x u\|_{L^2(\mathbb{R}^2)} + \|(-\Delta)_{y}^{s/2}u\|_{L^2(\mathbb{R}^2)} < \infty,
\]
where $(-\Delta)_{y}^{s/2}$ is a fractional Laplacian operator with respect to the $y$ variable, defined as
\[
(-\Delta)_{y}^{s/2}u(x,y) = C_s~ \text{P.V.} \int_{\mathbb{R}^2} \frac{u(x,y) - u(x,z)}{|y-z|^{2+s}} \, \mathrm{d}z,
\]
with $C_s$ being a normalizing constant and P.V. standing for the principal value.

The space $\mathcal{H}^{1,s}(\mathbb{R}^2)$ is a Banach space with respect to the norm:
\[
\|u\|_{\mathcal{H}^{1,s}}^2 := \|u\|_{L^2}^2 + \|\partial_x u\|_{L^2}^2 + \|(-\Delta)_{y}^{s/2} u\|_{L^2}^2.
\]
The product space $\mathbb{D} = \mathcal{H}^{1,s}(\mathbb{R}^2) \times \mathcal{H}^{1,s}(\mathbb{R}^2)$ is equipped with the norm:
\[
\|(u,v)\|_{\mathbb{D}}^2 := \|u\|_{\mathcal{H}^{1,s}}^2 + \|v\|_{\mathcal{H}^{1,s}}^2.
\]

 In this work, we address several significant challenges associated with the study of normalized solutions for systems involving anisotropic operators. One of the major novelties is the use of the fractional Laplacian with respect to the $y$-variable, which introduces anisotropy into the system. This anisotropy complicates the analysis due to the directional dependence of the diffusion operator, requiring techniques that extend beyond the traditional methods used for isotropic or purely fractional operators.
A key difficulty we overcame is the analysis of the Sobolev-Liouville space $\mathcal{H}^{1,s}(\mathbb{R}^2)$, which combines both local and non-local behaviors. This makes the estimation of norms, particularly in proving coercivity, highly non-trivial. Furthermore, the presence of mixed nonlinearities with concave-convex terms, along with $L^2$ constraints, introduces additional challenges in constructing suitable functionals and in the application of variational methods, such as the mountain pass theorem.

\subsection*{Paper Organization}  The paper is organized as follows. In Section \ref{Preliminaries}, we present the hypotheses and main results, including the introduction of the energy functional under normalizing constraints, the Pohozaev manifold, and the main theorems. We also discuss preliminary concepts, such as ground state solutions and the scalar problem. In Section \ref{sec-3}, we focus on the Pohozaev manifold, coercivity, and the use of Steiner rearrangement techniques. Section \ref{sec:Palais_Smale} is devoted to proving the existence of a special Palais-Smale sequence and obtaining estimates for the constant $C_{(a,b)}$. Finally, Section \ref{sec:PS-sequence-sec-5} presents the proof of the main theorem \ref{th:main-t1}.

\section{Preliminaries and Main Results}\label{Preliminaries}
To state our primary findings, let us first review some definitions and existing facts.
The energy functional associated with the system \eqref{eq:20220902-maine1} is defined by 
\begin{align} \label{energy functional}
	\begin{split}
		J(u,v) &= \frac{1}{2}\int_{\mathbb{R}^{2}} \big(|\partial_{x}u|^2 +  |(-\Delta)_{y}^{s /2} u|^2 \big)\,\mathrm{d}x \mathrm{d}y +\frac{1}{2}\int_{\mathbb{R}^{2}} \big(|\partial_{x}v|^2 +  |(-\Delta)_{y}^{s /2} v|^2  \big)\,\mathrm{d}x \mathrm{d}y  \\
		&\quad - \frac{\mu _1}{p}\int_{\mathbb{R}^{2}} |u|^{p}\, \mathrm{d}x \mathrm{d}y -\frac{\mu _2}{q}\int_{\mathbb{R}^{2}} |v|^{q}\, \mathrm{d}x \mathrm{d}y  - \beta \int_{\mathbb{R}^{2}}|u|^{r_1}|v|^{r_2}\, \mathrm{d}x \mathrm{d}y.
	\end{split}
\end{align}
We consider the following sets corresponding to our normalization constraint 
$$ S_a=\left\{u \in \mathcal{H}^{1,s}(\mathbb{R}^2) :\int_{\mathbb{R}^2} u^2 \,\mathrm{d}x \mathrm{d}y =a\right\}~~\text{and}~~S_b=\left\{u \in \mathcal{H}^{1,s}(\mathbb{R}^2) :  \int_{\mathbb{R}^2} v^2 \,\mathrm{d}x \mathrm{d}y =b\right\}.$$

The above functional $J$ is well-defined and $C^1$ on the product space $\mathbb{D}$. For  $(u_0,v_0) \in \mathbb{D}$, the Fr\'{e}chet derivative of $J$ at $(u,v) \in \mathbb{D}$ can be represented as
\begin{align*}
	\langle J'(u,v) | (u_0,v_0)\rangle &= \int_{\mathbb{R}^{2}}  \big(  \partial_x u  \partial_x u_0 +(-\Delta)_{y}^{s/2}u  (-\Delta)_{y}^{s/2}u_0   \big)\,\mathrm{d}x\mathrm{d}y \\
	&\qquad + \int_{\mathbb{R}^{2}}  \big(\partial_x v  \partial_x v_0 +(-\Delta)_{y}^{s/2}v  (-\Delta)_{y}^{s/2}v_0   \big)\,\mathrm{d}x\mathrm{d}y  \\ 
	& \quad- \mu _1\int_{\mathbb{R}^{2}} |u|^{p-2} u  u_0 \,\mathrm{d}x\mathrm{d}y
	- \mu _2\int_{\mathbb{R}^{2}} |v|^{q-2} v  v_0 \,\mathrm{d}x\mathrm{d}y\\
	&~~~~~ - \beta r_1 \int_{\mathbb{R}^{2}}|u|^{r_1 -2}u u_0|v|^{r_2}\,\mathrm{d}x\mathrm{d}y
	- \beta r_2 \int_{\mathbb{R}^{2}}|u|^{r_1}|v|^{r_2-2}v  v_0 \,\mathrm{d}x\mathrm{d}y,
\end{align*}
where $J'(u,v)$ is the Fr\'{e}chet derivative of $J$ at $(u,v)$, and the duality bracket between the dual $\mathbb{D}^*$ and product space $\mathbb{D}$  is represented as ${\langle}  \cdot,\cdot \rangle$.

\begin{definition}[\textbf{Pohozaev Manifold}]
The following natural constraint determines the minimum of the energy functional J which represents the ground state solution of the above problem
\begin{align}\label{2.7}
    \mathcal P = \{(u,v)\in \mathbb{D} \setminus \{(0,0)\}: P(u,v)=0\},
\end{align}
where
\begin{align}\label{2.8}
   \begin{split}
		P(u,v) &= s\int_{\mathbb{R}^{2}} \big(|\partial_{x}u|^2 +  |(-\Delta)_{y}^{s /2} u|^2  \big)\,\mathrm{d}x \mathrm{d}y +s\int_{\mathbb{R}^{2}} \big(|\partial_{x}v|^2 +  |(-\Delta)_{y}^{s /2} v|^2 \big)\,\mathrm{d}x \mathrm{d}y  \\
		&\quad - \frac{(1+s)(p-2)}{2p}\mu _1\int_{\mathbb{R}^{2}} |u|^{p}\, \mathrm{d}x \mathrm{d}y -\frac{(1+s)(q-2)}{2q}\mu _2\int_{\mathbb{R}^{2}} |v|^{q}\, \mathrm{d}x \mathrm{d}y  \\
  &\quad-{\frac{(1+s)(r_1+r_2-2)}{2}} \beta \int_{\mathbb{R}^{2}}|u|^{r_1}|v|^{r_2}\, \mathrm{d}x \mathrm{d}y
	\end{split}
 \end{align}
As evident, $P$ belongs to the class ${C}^{1}$ and ${P}^{'}$ maps bounded sets to bounded sets.
The boundedness of the functional $J(u,v)$ on $S_a \times S_b$ fails when $p > \frac{2(1+3s)}{1+s}$. For $(u,v) \in S_a \times S_b$ and $t > 0$, we introduce a scaling argument similar to \ref{eq:ut-invariant}. Consequently, we have $(u_t, v_t) \in S_a \times S_b$, and the energy functional becomes
\[
\begin{split}
	J(u_t,v_t)& = \frac{1}{2}t^{2s}\int_{\mathbb{R}^{2}} \left( |\partial_{x}u|^2 + |(-\Delta)_{y}^{s/2} u|^2 \right)\, \mathrm{d}x \mathrm{d}y + \frac{1}{2}t^{2s}\int_{\mathbb{R}^{2}} \left( |\partial_{x}v|^2 + |(-\Delta)_{y}^{s/2} v|^2 \right)\, \mathrm{d}x \mathrm{d}y\\&\quad
 - t^{\frac{(1+s)(p-2)}{2}} \frac{\mu_1}{p} \int_{\mathbb{R}^{2}} |u|^p \, \mathrm{d}x \mathrm{d}y - t^{\frac{(1+s)(q-2)}{2}} \frac{\mu_2}{q} \int_{\mathbb{R}^{2}} |v|^q \, \mathrm{d}x \mathrm{d}y\\&\quad 
- t^{\frac{(1+s)(r_1 + r_2 - 2)}{2}} \beta \int_{\mathbb{R}^{2}} |u|^{r_1} |v|^{r_2} \, \mathrm{d}x \mathrm{d}y.\end{split}
\]
As $t \to +\infty$, the energy functional $J$, constrained on $S_a \times S_b$, becomes unbounded from below for any $a,b > 0$. Hence, a global minimization approach becomes ineffective in the supercritical case.

Inspired by the minimization method on the Pohozaev manifold, we aim to construct a submanifold within $S_a \times S_b$ where the functional $J(u,v)$ is bounded from below and coercive. Our goal is to find minimizers of the energy functional $E(u)$ on this submanifold, denoted by $\mathcal{P}_{(a,b)}$. 

Additionally, our analysis shows that the energy functional $J|_{S_a \times S_b}$ exhibits a mountain pass geometry. This observation motivates the application of the mountain pass theorem to identify a critical point of the functional. This approach is built upon an advanced version of the min-max principle, initially developed by Ghoussoub~\cite{MR1251958}. 

All critical points of $J$, subject to the constraint ${S_a} \times {S_b}$, lie in $\mathcal{P}$. This constraint originates from the Pohozaev identity. To establish the existence of a ground state solution, we prove the existence of a critical point of $J$ such that 
\[
\inf_{\mathcal{P} \cap (S_a \times S_b)} J(u,v)
\] is achieved.

\end{definition}
Prior to working over the original problem, let us first look into a relaxed problem. Let us introduce the set for $a> 0$ as
\begin{equation*}
D_a:=\{u\in \mathcal{H}^{1,s}(\R^2): \|u\|_2^2\leq a\}
\end{equation*}
and for $a>0$, $b>0$ we define
\begin{equation*}
\mathcal{P}_{(a,b)}:=\mathcal{P}\cap (D_a\times D_b).
\end{equation*}
The relaxed problem is to look for a non$-$trivial critical point $ (u,v)\in$ $\mathbb{D}$ of J restricted to $D_a \times D_b\nonumber\\$ at the level 
\begin{equation}\label{2.6}
    C_{(a,b)}:=\inf_{\mathcal{P}_{(a,b)}}J(u,v).
\end{equation}
Once we find a solution to the relaxed problem then we consider the original problem.
\medskip

Following is our main result.
\begin{theorem}\label{th:main-t1}
Let $\frac{2(1+3s)}{1+s}<p,q,r_1+r_2<2_s=\frac{2(1+s)}{1-s}$ and $r_1,r_2>1$. Then the following holds true.
\begin{itemize}
\item[(i)] For $a>0$ and $b\in [b_{p,q,\mu_1,\mu_2,a},+\infty)$. If
$$\begin{cases}
r_1<2\\
\beta>0
\end{cases}~\hbox{or}~
\begin{cases}
r_1=2\\
\beta>\beta_{q,\mu_2,b,r_2} \, ,
\end{cases}$$
then there exists a ground state solution $(\lambda_1,\lambda_2,u,v)$ to  \eqref{eq:20220902-maine1}. In addition, $\lambda_1>0,\lambda_2>0$ and $u, v $ are Steiner symmetric functions.
\item[(ii)]Let $a>0$ and $b\in (0,b_{p,q,\mu_1,\mu_2,a}]$. If
$$\begin{cases}
r_2<2\\
\beta>0
\end{cases}~ \hbox{or}~
\begin{cases}
r_2=2\\
\beta>\beta_{p,\mu_1,a,r_1} \, ,
\end{cases}$$
then there exists a ground state solution $(\lambda_1,\lambda_2,u,v)$ to   \eqref{eq:20220902-maine1}. In addition, $\lambda_1>0,\lambda_2>0$ and $u, v $ are Steiner symmetric functions.
\end{itemize}
\end{theorem}

\begin{corollary}\label{cro:main-c2}
Let $\frac{2(1+3s)}{1+s}<p,q,r_1+r_2$ and $r_1>1,r_2>1$. Then
\begin{itemize}
\item[(i)] for any $a>0$ and $\beta>0$, there exists a ground state solution $(\lambda_1,\lambda_2,u,v)$ of   \eqref{eq:20220902-maine1} provided $b\in[b_{p,q,\mu_1,\mu_2,a},+\infty)$ and $r_1\leq 2$.  In addition, $\lambda_1>0,\lambda_2>0$ and $u, v $ are Steiner symmetric functions.
\item[(ii)]for any $a>0$ and $\beta>0$, there exists a positive ground state solution $(\lambda_1,\lambda_2,u,v)$ to   \eqref{eq:20220902-maine1} provided  $b\in(0,b_{p,q,\mu_1,\mu_2,a}]$ and $r_2\leq 2$. In addition, $\lambda_1>0,\lambda_2>0$ and $u, v $ are Steiner symmetric functions.
\end{itemize}
\end{corollary}
To attain the further characteristics of the ground states discussed in Theorem \eqref{th:main-t1} and its corollary, we take advantage of the fact that any minimum for $$\inf_{\mathcal{P} \cap (S_a \times S_b)}J(u,v)$$ is a critical point for $J$ constrained to $S_a \times S_b$ and is attained by Steiner symmetric functions. 

\subsection{ Preliminaries}\label{sec:prelim}
\renewcommand{\theequation}{2.\arabic{equation}}
This section presents some preliminary observations. \begin{definition}[\textbf{Ground state solution}] \label{def of gss}
	 A non-trivial critical point $(u_1,v_1) \in \mathbb{D} \backslash \{(0, 0)\}$ of $J$ over $\mathbb{D}$ is said to be a ground state solution of \eqref{eq:20220902-maine1} if its energy is minimal among all the non-trivial critical points i.e.
	\begin{equation} \label{ground state level}
		 J(u_{1},v_{1}) = \min \{ J(u,v): (u, v) \in \mathbb{D} \setminus \{(0, 0)\} ~\text{and}~ J'(u,v)=0 \}.
	\end{equation}  
\end{definition}
We recall that there exists a solution to the following scalar problem, for $p \in (2,2_s)$
\begin{equation}
\begin{cases} \label{2.1}
-\partial_{xx}u + (-\Delta)_{y}^s u +  u  =  u^{p-1} ~~~~~~~\text{in} ~ \mathbb{R}^{2},\\
~~u>0 ~~~~~~~~~~~~~~~~~~~~~~~~~~~~~~~~~~~~\text{in} ~\mathbb{R}^{2}
\end{cases} 
\end{equation}
actually it has a unique ground state solution in $ \mathcal{H}^{1,s}(\mathbb{R}^2)$, cf. \cite{gou2023solitary}. For fixed $\mu$ $\in \mathbb{R}^+$ and $\lambda >0$, we consider the solution pair $(\lambda ,w) \in \mathbb{R^+} \times \mathcal{H}^{1,s}(\mathbb{R}^2)$ of
\begin{equation}
\begin{cases}\label{2.2}
 -\partial_{xx}w + (-\Delta)_{y}^s w + \lambda  w  = \mu  w^{p-1} ~~~~~~~\text{in} ~ \mathbb{R}^{2},\\
 ~~w>0 ~~~~~~~~~~~~~~~~~~~~~~~~~~~~~~~~~~~~~~~~~\text{in} ~\mathbb{R}^{2}
\end{cases}
\end{equation}
such that $\int_{\mathbb{R}^2} w^2 \,\mathrm{d}x\mathrm{d}y =a $.
We define the energy functional  corresponding to \eqref{2.2} as 
$$F_{p,\mu} = \frac{1}{2}\int_{\mathbb{R}^{2}} \big(|\partial_{x}w|^2 +  |(-\Delta)_{y}^{s /2} w|^2 \big)\mathrm{d}x \mathrm{d}y-\frac{\mu}{p}\int_{\mathbb{R}^{2}} |w|^p  \,\mathrm{d}x \mathrm{d}y $$
and the unique solution corresponding to   \eqref{2.2} is formally given by 
\begin{equation}\label{2.5}
    z_{p,\mu,a}(x)=\Big(\frac{\lambda}{\mu}\Big)^{\frac{1}{p-2}}u_p(\lambda^\frac{1}{2}x,\lambda^\frac{1}{2s}y)
\end{equation}
where $u_p(x,y)$ is a solution of \ref{2.1} and
\begin{equation}\label{2.4}
\lambda=a^{\frac{-2s(p-2)}{(1+s)(p-2)-4s}}\|u_p\|_2^{\frac{4s(p-2)}{(1+s)(p-2)-4s}}\mu^{\frac{-4s}{(1+s)(p-2)-4s}}
\end{equation}
and $p\in (2,2_s)\backslash \Big\{\frac{2(1+3s)}{1+s}\Big\}$. 

Let us define the set
\begin{equation*}
\tilde{\mathcal{P}}_{p,\mu}:=\left\{w\in \mathcal{H}^{1,s}(\mathbb{R}^{2})\setminus \{0\}:s \int_{\mathbb{R}^{2}} \big(|\partial_{x}w|^2 +  |(-\Delta)_{y}^{s /2} w|^2\big)\,\mathrm{d}x \mathrm{d}y=\frac{(1+s)(p-2)}{2p}\mu\int_{\mathbb{R}^{2}} |w|^{p}\, \mathrm{d}x \mathrm{d}y\right\},
\end{equation*}
where $\mu$ is as in \eqref{2.2} and for $a>0$, consider $\tilde{\mathcal{P}}_{p,\mu,a}:=\tilde{\mathcal{P}}_{p,\mu}\cap S_a$.
Recalling the fiber map
\begin{equation}\label{eq:ut-invariant}
w_t(x,y) :=t^{\frac{1+s}{2}}w(t^{s}x,ty)
\end{equation}
for $t\in \R^+$, which keeps the $L^2$-norm preserved, we get the following results.
\begin{lemma}\label{lemma3.1}
\begin{itemize}
\item[(i)] The unique solution of \eqref{2.2} belongs to $ \tilde{\mathcal{P}}_{p,\mu,a}$.  Moreover, $F_{p,\mu}$ is minimized on $\tilde{\mathcal{P}}_{p,\mu,a}$ by the unique solution to \eqref{2.2}.
\item[(ii)]If $p\neq \frac{2(1+3s)}{1+s}$, then for any $w\neq 0$, there exists a unique $l=l(w)>0$ such that $w_l\in \tilde{\mathcal{P}}_{p,\mu}$ and it holds that
\begin{equation}
\label{eq:20210622-xe1}
F_{p,\mu}(w_l)=\begin{cases}
\max_{l>0} F_{p,\mu}(w_l),&\hbox{if}~\frac{2(1+3s)}{1+s}<p<2_s,\\
\min_{l>0} F_{p,\mu}(w_l), &\hbox{if}~2<p<\frac{2(1+3s)}{1+s}.
\end{cases}
\end{equation}
\item[(iii)]Let $z_{p,\mu,a}$ be defined by \eqref{2.5}. Then
\begin{align*}
m_{p,\mu,a}:=&F_{p,\mu}(z_{p,\mu,a})=\inf_{w\in \tilde{\mathcal{P}}_{p,\mu,a}}F_{p,\mu}(w)\\
=&\begin{cases}
    \inf_{w\in S_a}\max_{l>0} F_{p,\mu}(w_l),&\hbox{if}~\frac{2(1+3s)}{1+s}<p<2_s,\\
    \inf_{w\in S_a}\min_{l>0} F_{p,\mu}(w_l)=\inf_{w\in S_a}F_{p,\mu}(w),&\hbox{if}~2<p<\frac{2(1+3s)}{1+s}.
    \end{cases}
\end{align*}
\end{itemize}
\end{lemma}
\begin{proof}
 (i) The proof of the first part is a simple consequence of Pohozaev identity \eqref{2.8} and for the second part, we refer \cite[Theorem 2.1]{MR3539467}.\\
 (ii) Let $w \neq 0$ and $l>0$ then a straightforward computation yields
\begin{equation*}
    f(l)= F_{p,\mu}(w_l)= \frac{l^{2s}}{2}\int_{\mathbb{R}^{2}} \big(|\partial_{x}w|^2 +  |(-\Delta)_{y}^{s /2} w|^2 \big)\mathrm{d}x \mathrm{d}y-l^{\frac{(1+s)(p-2)}{2}}\frac{\mu}{p}\int_{\mathbb{R}^{2}} |w|^p  \,\mathrm{d}x \mathrm{d}y.
\end{equation*}
We can see that  $w_l\in \tilde{\mathcal{P}}_{p,\mu}$ if and only if '$l$' is a critical point of $f$. It is easy to verify that when $p \neq \frac{2(1+3s)}{(1+s)}$, following is a critical point of $f$:
\begin{equation*}
    l_w=\bigg(\frac{\int_{\mathbb{R}^{2}} \big(|\partial_{x}w|^2 +  |(-\Delta)_{y}^{s /2} w|^2 \big)\mathrm{d}x \mathrm{d}y}{\frac{(1+s)(p-2)}{2sp}{\mu}\int_{\mathbb{R}^{2}} |w|^p  \,\mathrm{d}x \mathrm{d}y} \bigg)^{\frac{2}{(1+s)(p-2)-4s}}
\end{equation*}
and 
\begin{equation*}
    \begin{split}
        f(l_w)&= \frac{(1+s)(p-2)-4s}{4sp}\bigg({\int_{\mathbb{R}^{2}}}\big(|\partial_{x}w|^2+  |(-\Delta)_{y}^{s /2} w|^2 \big)\mathrm{d}x \mathrm{d}y \bigg)^{\frac{(1+s)(p-2)}{(1+s)(p-2)-4s}}\bigg(\mu \int_{\mathbb{R}^{2}} |w|^p  \,\mathrm{d}x \mathrm{d}y\bigg)^{\frac{-4s}{(1+s)(p-2)-4s}} \\
     & \quad \quad\bigg(\frac{2sp}{(1+s)(p-2)}\bigg)^{\frac{(1+s)(p-2)}{(1+s)(p-2)-4s}}
    \end{split}
\end{equation*}
In addition, the maximum is reached by $f(l)>0$ if $p>\frac{2(1+3s)}{1+s}$, and the minimum is reached by $f(l)<0$ if $p<\frac{2(1+3s)}{1+s}$. Hence,  \eqref{eq:20210622-xe1} holds.\\
(iii) Since $z_{p,\mu,a}$ minimizes $F_{p,\mu}$. Hence $F_{p,\mu}$ achieves its infimum, and by (i) and (ii) we can say that if $\frac{2(1+3s)}{1+s}<p<2_s$ then $m_{p,\mu,a}=\inf_{w\in S_a}\max_{l>0} F_{p,\mu}(w_l)$ and if $2<p<\frac{2(1+3s)}{1+s}$ then $$m_{p,\mu,a}=\inf_{w\in S_a}\min_{l>0} F_{p,\mu}(w_l).$$ 
This completes the proof.
 \end{proof}
\begin{lemma}\label{lemma3.2}
\begin{enumerate}
\item[(i)]If $p\in \left(\frac{2(1+3s)}{1+s},\frac{2(1+s)}{1-s}\right)$, then $m_{p,\mu,a}$ is a continuous, strictly decreasing function with respect to $a\in \R^+$ and
 $$\lim_{a\rightarrow 0^+}m_{p,\mu,a}=+\infty \;\hbox{and}\;\lim_{a\rightarrow +\infty}m_{p,\mu,a}=0.$$
\item[(ii)]If $p\in \left(2,\frac{2(1+3s)}{1+s}\right)$, then $m_{p,\mu,a}$ is a continuous, strictly decreasing function with respect to $a\in \R^+$ and
 $$\lim_{a\rightarrow 0^+}m_{p,\mu,a}=0 \;\hbox{and}\;\lim_{a\rightarrow +\infty}m_{p,\mu,a}=-\infty.$$
\end{enumerate}
\end{lemma}
\begin{proof}
By Lemma \ref{lemma3.1}-(iii), we know that $m_{p,\mu,a}=F_{p,\mu}(z_{p,\mu,a})$ which gives
\begin{align*}
 {m_{p,\mu,a}=\Big(\frac{(1+s)(p-2)-4s}{4sp}\Big)\mu^{\frac{-4s}{(1+s)(p-2)-4s}}\|u_p\|_2^{\frac{4sp-2(1+s)(p-2)}{(1+s)(p-2)-4s}}a^{\frac{-2sp+(1+s)(p-2)}{(1+s)(p-2)-4s}}\int_{\mathbb{R}^{2}}|u|^p\mathrm{d}x \mathrm{d}y}  .
\end{align*}
  If $\frac{2(1+3s)}{1+s}<p<\frac{2(1+s)}{1-s}$, we see that the exponent ${\frac{-2sp+(1+s)(p-2)}{(1+s)(p-2)-4s}}<0$, hence $m_{p,\mu,a}$ strictly decreases with respect to $a\in \R^+$. Furthermore,
 $$\lim_{a\rightarrow 0^+}m_{p,\mu,a}=+\infty \;\hbox{and}\;\lim_{a\rightarrow +\infty}m_{p,\mu,a}=0.$$
 Similarly, if $2<p<\frac{2(1+3s)}{1+s}$, then ${\frac{-2sp+(1+s)(p-2)}{(1+s)(p-2)-4s}}>0$, and noting that
 $\frac{(1+s)(p-2)-4s}{4sp}<0.$ Thus in this scenario as well, $m_{p,\mu,a}$ strictly decreases with respect to $a\in \R^+$. Additionally,
 $$\lim_{a\rightarrow 0^+}m_{p,\mu,a}=0 \;\hbox{and}\;\lim_{a\rightarrow +\infty}m_{p,\mu,a}=-\infty.$$
 This finishes the proof.
\end{proof}
Now for fixed $p,q,\mu_1,\mu_2$ as given in \eqref{eq:20220902-maine1} and for $a>0$, let us define
$$\alpha=\frac{\mu_2{((q-2)(1+s)-4s)(2sp-(1+s)(p-2))}}{\mu_1{((p-2)(1+s)-4s)}{(2sq-(1+s)(q-2))}}  {}$$\\
and
\begin{align*}
 b_{p,q,\mu_1,\mu_2,a}:=&\left(\alpha\right)^{\frac{(q-2)(1+s)-4s}{2sq-(1+s)(q-2)}}
\|u_q\|_{2}^{\frac{4(q-2)}{2sq-(1+s)(q-2)}}   \|u_p\|_{2}^{-\frac{4(p-2)}{2sq-(1+s)(q-2)} \frac{(q-2)(1+s)-4s}{(p-2)(1+s)-4s}} a^{{{\frac{(q-2)(1+s)-4s}{(p-2)(1+s)-4s} \frac{2sp-(1+s)(p-2)}{2sq-(1+s)(q-2)}}}}.
\end{align*}
and
\begin{equation}\label{def:-best-const}
 \beta_{p,\mu,a,r}:=\frac{1}{2}\inf_{h\in \mathcal{H}^{1,s}(\R^2)\backslash\{0\}}\frac{{\Big[{\int_{\mathbb{R}^{2}} \big(|\partial_{x}h|^2 +  |(-\Delta)_{y}^{s /2} h|^2 \big)\,\mathrm{d}x \mathrm{d}y }\Big]}}{\int_{\R^2}|z_{p,\mu,a}|^{r} |h|^2 \mathrm{d}x\mathrm{d}y}.
\end{equation}

\section{ Pohozaev type constraint}\label{sec-3}
\renewcommand{\theequation}{3.\arabic{equation}}

\begin{definition}[\textbf{Pohozaev Manifold}]
The following natural constraint determines the minimum of the energy functional J which represents the ground state solution of the above problem
\begin{align}
    \mathcal P = \{(u,v)\in \mathbb{D} \setminus \{(0,0)\}: P(u,v)=0\},
\end{align}
where
\begin{align}
   \begin{split}
		P(u,v) &= s\int_{\mathbb{R}^{2}} \big(|\partial_{x}u|^2 +  |(-\Delta)_{y}^{s /2} u|^2  \big)\,\mathrm{d}x \mathrm{d}y +s\int_{\mathbb{R}^{2}} \big(|\partial_{x}v|^2 +  |(-\Delta)_{y}^{s /2} v|^2 \big)\,\mathrm{d}x \mathrm{d}y  \\
		&\quad - \frac{(1+s)(p-2)}{2p}\mu _1\int_{\mathbb{R}^{2}} |u|^{p}\, \mathrm{d}x \mathrm{d}y -\frac{(1+s)(q-2)}{2q}\mu _2\int_{\mathbb{R}^{2}} |v|^{q}\, \mathrm{d}x \mathrm{d}y  \\
  &\quad-{\frac{(1+s)(r_1+r_2-2)}{2}} \beta \int_{\mathbb{R}^{2}}|u|^{r_1}|v|^{r_2}\, \mathrm{d}x \mathrm{d}y
	\end{split}
 \end{align}
 \end{definition}
As evident, $P$ belongs to the class ${C}^{1}$ and ${P}^{'}$ maps bounded sets to bounded sets.  Let us introduce the set for $a> 0$ as
\begin{equation*}
D_a:=\{u\in \mathcal{H}^{1,s}(\R^2): \|u\|_2^2\leq a\}
\end{equation*}
and for $a>0$, $b>0$ we define
\begin{equation*}
\mathcal{P}_{(a,b)}:=\mathcal{P}\cap (D_a\times D_b).
\end{equation*}
We prove that the constraint $\mathcal{P}$ does not yield a Lagrange parameter in this section, among other aspects of the constraint identified. \textit{Assume that $\frac{2(1+3s)}{1+s}<p,q,r_1+r_2<2_s=\frac{2(1+s)}{1-s}$ throughout this article}. To begin with, following \cite[Theorem 2.2]{esfahani2024new}, we can establish the result below.
 \begin{lemma}\label{lemma 4.1}
Every non-negative solution  $(u,v)\in \mathbb{D}\backslash\{(0,0)\}$ of
\begin{equation}{\label{4.1}}
    \left\{
     \begin{aligned}
         -\partial_{xx}u + (-\Delta)_{y}^{s} u + \lambda _1 u  = \mu _1 u^{p-1}+ \beta r_1 u^{r_1-1} v^{r_2},  \\
            -\partial_{xx}v + (-\Delta)_{y}^{s} v +\lambda _2 v  = \mu _2 v^{q-1} + \beta r_2 u^{r_1}v^{r_2-1}  ,
    \end{aligned}
    \right.
\end{equation}
belongs to $ \mathcal{P}$, where $\mathcal{P}$ is the Pohozaev manifold defined in \eqref{2.7}.
 \end{lemma} 
Let us now define $\phi_{(u,v)}:\R^+\rightarrow \R$ for every fixed $(u,v)\neq (0,0)$ $\in$ $\mathbb{D}$ by
\begin{align*}
\phi_{(u,v)}(t):=J(u_t,v_t)=&\frac{1}{2}t^{2s}\int_{\mathbb{R}^{2}} \big(|\partial_{x}u|^2 +  |(-\Delta)_{y}^{s /2} u|^2 \big)\,\mathrm{d}x \mathrm{d}y +\frac{1}{2}t^{2s}\int_{\mathbb{R}^{2}} \big(|\partial_{x}v|^2 +  |(-\Delta)_{y}^{s /2} v|^2  \big)\,\mathrm{d}x \mathrm{d}y  \\
		&\quad - t^{\frac{(1+s)(p-2)}{2}}\frac{\mu _1}{p}\int_{\mathbb{R}^{2}} |u|^{p}\, \mathrm{d}x \mathrm{d}y -t^{\frac{(1+s)(q-2)}{2}}\frac{\mu _2}{q}\int_{\mathbb{R}^{2}} |v|^{q}\, \mathrm{d}x \mathrm{d}y\\
             & - t^{\frac{(1+s)(r_1+r_2-2)}{2}}{\beta} \int_{\mathbb{R}^{2}}|u|^{r_1}|v|^{r_2}\, \mathrm{d}x \mathrm{d}y
\end{align*}
A direct calculation shows that
\begin{align}\label{4.2}
\phi'_{(u,v)}(t)=&\frac{d}{dt} J(u_t,v_t)\nonumber\\
=&{s}t^{2s-1}\int_{\mathbb{R}^{2}} \big(|\partial_{x}u|^2 +  |(-\Delta)_{y}^{s /2} u|^2 \big)\,\mathrm{d}x \mathrm{d}y +{s}t^{2s-1}\int_{\mathbb{R}^{2}} \big(|\partial_{x}v|^2 +  |(-\Delta)_{y}^{s /2} v|^2  \big)\,\mathrm{d}x \mathrm{d}y  \nonumber\\
		&\quad - \frac{(1+s)(p-2)}{2}t^{\frac{(1+s)(p-2)-2}{2}}\frac{\mu _1}{p}\int_{\mathbb{R}^{2}} |u|^{p}\, \mathrm{d}x \mathrm{d}y -\frac{(1+s)(q-2)}{2}t^{\frac{(1+s)(q-2)-2}{2}}\frac{\mu _2}{q}\int_{\mathbb{R}^{2}} |v|^{q}\, \mathrm{d}x \mathrm{d}y \nonumber\\
                  &\quad                    - \frac{(1+s)(r_1+r_2-2)}{2}t^{\frac{(1+s)(r_1+{r_2}-2)-2}{2}}{\beta} \int_{\mathbb{R}^{2}}|u|^{r_1}|v|^{r_2}\, \mathrm{d}x \mathrm{d}y \nonumber\\ 
\end{align}
which implies $P(u,v)=\phi'_{(u,v)}(1)$ and $P(u_t,v_t)=t \phi'_{(u,v)}(t)$. As a consequence of this, we have the following result.
\begin{lemma}\label{lemma4.2}
For every $(u,v)\neq (0,0)$ $\in$ $\mathbb{D}$, $t\in \R^+$ is a critical point of $\phi_{(u,v)}(t)$ if and only if $(u_t,v_t)\in \mathcal{P}$.
\end{lemma}
\begin{corollary}\label{corollary4.3}
    For any $(0,0)\neq (u,v)\in D_a\times D_b$, there exists a unique $t_0=t_{(u,v)}>0$ such that $(u_t,v_t)\in \mathcal{P}_{(a,b)}$. Furthermore, $t_{(u,v)}<(resp.~=,>) 1$ if and only if $P(u,v)<(resp.~=,>) 0$.
\end{corollary}
\begin{proof}
For any $t>0$, since $(u,v)\in D_a\times D_b$ we have that $( u_t, v_t)\in D_a\times D_b$, recalling that the fiber map \eqref{eq:ut-invariant} retains the $L^2$-norm. Then by Lemma \ref{lemma4.2}, $(u_t,v_t)\in \mathcal{P}_{(a,b)}$ if and only if $\phi'_{(u,v)}(t)=0$. Now by \eqref{4.2} and $\frac{2(1+3s)}{1+s}<p,q,r_1+r_2<2_s$, we obtain a unique  $t_0=t_{(u,v)}$ satisfying $\phi_{(u,v)}(t_0)=\max_{t>0}\phi_{(u,v)}(t)$. Also
$$\phi'_{(u,v)}(t)>0\;\hbox{for}~0<t<t_{(u,v)}\;\hbox{while}~\phi'_{(u,v)}(t)<0\;\hbox{for}~t>t_{(u,v)}.$$
So combining with $P(u,v)=\phi'_{(u,v)}(1)$, we obtain that
$$P(u,v)<(resp.~=,>)0\Leftrightarrow \phi'_{(u,v)}(1)<(resp.~=,>)0\Leftrightarrow t_{(u,v)}<(resp.~=,>)1.$$
This finishes the proof.
\end{proof}
\begin{lemma}\label{lemma4.4}
For every $(u,v)$ satisfying $P(u,v)=0$, there exists some  $C_0>0$  depending only on $p,q,r_1,r_2$,  such that
\begin{equation*}
J(u,v)\geq C_0 \Big[ \int_{\mathbb{R}^{2}} \big(|\partial_{x}u|^2 +  |(-\Delta)_{y}^{s /2} u|^2 \big)\,\mathrm{d}x \mathrm{d}y +\int_{\mathbb{R}^{2}} \big(|\partial_{x}v|^2 +  |(-\Delta)_{y}^{s /2} v|^2  \big)\,\mathrm{d}x \mathrm{d}y \Big] .
\end{equation*}
\end{lemma}
\begin{proof}
For~$\frac{2(1+3s)}{1+s}<p,q,r_1+r_2<2_s $, we
set
\begin{equation*}
\tau:=\max\left\{\frac{2s}{(p-2)(1+s)}, \frac{2s}{(q-2)(1+s)}, \frac{2s}{(r_1+r_2-2)(1+s)}\right\},
\end{equation*}
then one can see that $0<\tau<\frac{1}{2}$.
By the definition of $\tau$,
\begin{align}
\frac{\mu_1}{p}\|u\|_p^p&+\frac{\mu_2}{q}\|v\|_q^q+\beta \int_{\R^2} |u|^{r_1}|v|^{r_2}\mathrm{d}x\mathrm{d}y\nonumber\\
\leq &\tau \left\{\frac{(p-2)(1+s)}{2sp}\mu_1\|u\|_p^p+\frac{(q-2)(1+s)}{2sq}\mu_2\|v\|_q^q
+\frac{(r_1+r_2-2)(1+s)}{2s}\beta \int_{\R^2} |u|^{r_1} |v|^{r_2} \mathrm{d}x\right\} \nonumber
\end{align}
So for any $(u,v)$ with $P(u,v)=0$,  we have that
\begin{equation*}
\begin{split}
&\frac{\mu_1}{p}\|u\|_p^p+\frac{\mu_2}{q}\|v\|_q^q+\beta \int_{\R^2} |u|^{r_1}|v|^{r_2}\mathrm{d}x\mathrm{d}y\\&\qquad
\leq \tau \Big[ \int_{\mathbb{R}^{2}} \big(|\partial_{x}u|^2 +  |(-\Delta)_{y}^{s /2} u|^2 \big)\,\mathrm{d}x \mathrm{d}y +\int_{\mathbb{R}^{2}} \big(|\partial_{x}v|^2 +  |(-\Delta)_{y}^{s /2} v|^2  \big)\,\mathrm{d}x \mathrm{d}y    \Big].
\end{split}
\end{equation*}
 Hence, we can take $C_0:=\frac{1}{2}-\tau>0$ such that
\begin{align*}
J(u,v)=&\frac{1}{2}\Big[ \int_{\mathbb{R}^{2}} \big(|\partial_{x}u|^2 +  |(-\Delta)_{y}^{s /2} u|^2 \big)\,\mathrm{d}x \mathrm{d}y +\int_{\mathbb{R}^{2}} \big(|\partial_{x}v|^2 +  |(-\Delta)_{y}^{s /2} v|^2  \big)\,\mathrm{d}x \mathrm{d}y    \Big]\\
&-\left[\frac{\mu_1}{p}\|u\|_p^p+\frac{\mu_2}{q}\|v\|_q^q+\beta \int_{\R^2} |u|^{r_1}|v|^{r_2}\mathrm{d}x\mathrm{d}y\right]\\
\geq &\frac{1}{2}\Big[ \int_{\mathbb{R}^{2}} \big(|\partial_{x}u|^2 +  |(-\Delta)_{y}^{s /2} u|^2 \big)\,\mathrm{d}x \mathrm{d}y +\int_{\mathbb{R}^{2}} \big(|\partial_{x}v|^2 +  |(-\Delta)_{y}^{s /2} v|^2  \big)\,\mathrm{d}x \mathrm{d}y    \Big]\\
&-\tau \Big[ \int_{\mathbb{R}^{2}} \big(|\partial_{x}u|^2 +  |(-\Delta)_{y}^{s /2} u|^2 \big)\,\mathrm{d}x \mathrm{d}y +\int_{\mathbb{R}^{2}} \big(|\partial_{x}v|^2 +  |(-\Delta)_{y}^{s /2} v|^2  \big)\,\mathrm{d}x \mathrm{d}y    \Big]\\
=&C_0\Big[ \int_{\mathbb{R}^{2}} \big(|\partial_{x}u|^2 +  |(-\Delta)_{y}^{s /2} u|^2 \big)\,\mathrm{d}x \mathrm{d}y +\int_{\mathbb{R}^{2}} \big(|\partial_{x}v|^2 +  |(-\Delta)_{y}^{s /2} v|^2  \big)\,\mathrm{d}x \mathrm{d}y    \Big].
\end{align*}
This finishes the proof.
\end{proof}
The following property is an immediate consequence of the above result.
\begin{corollary}\label{corollary4.5}
For any $a>0,b>0$,
$J\big|_{\mathcal{P}_{(a,b)}}$ is coercive.

\end{corollary}
We now prove that $C_{(a,b)}>0$ over $D_a \times D_b\nonumber$, where
\begin{equation}
    C_{(a,b)}:=\inf_{\mathcal{P}_{(a,b)}}J(u,v).
\end{equation}
It shall be a result of following Lemma.
\begin{lemma}\label{lemma4.6}
For any given $a,~b\geq 0$ with $(a,b)\neq (0,0)$, there exists some $\delta_{(a,b)}>0$ such that
\begin{equation*}
\inf_{(u,v)\in \mathcal{P}_{(a,b)}} \Big[{\int_{\mathbb{R}^{2}} \big(|\partial_{x}u|^2 +  |(-\Delta)_{y}^{s /2} u|^2 \big)\,\mathrm{d}x \mathrm{d}y +\int_{\mathbb{R}^{2}} \big(|\partial_{x}v|^2 +  |(-\Delta)_{y}^{s /2} v|^2  \big)\,\mathrm{d}x \mathrm{d}y   } \Big]\geq \delta_{(a,b)}.
\end{equation*}
\end{lemma}
\begin{proof}
Let us revisit the widely known Gagliardo-Nirenberg inequality \cite[lemma 2.1]{esfahani2024new} .
For $s \in (0, 1)$ and $2 \leq p < p_s = \frac{2(1+s)}{1-s}$, there exists a constant $C_{s,p} > 0$ such that
\begin{equation*}
{\int_{\mathbb{R}^{2}}|u|^p}\leq C_{s,p} \Big({\int_{\mathbb{R}^{2}} \big(|\partial_{x}u|^2}\Big)^{\frac{p-2}{4}} \Big({\int_{\mathbb{R}^{2}} |(-\Delta)_{y}^{s /2} u|^2 \,}\Big)^{\frac{p-2}{4s}}\Big({\int_{\mathbb{R}^{2}}|u|^2\Big)}^{\frac{p}{2}-\frac{(p-2)(1+s)}{4s}} ,\quad \forall{ u\in \mathcal{H}^{1,s}(\R^2)} \\.
\label{eq:main}
\end{equation*}
If $s > 1/2$, then $p = p_s$ is allowed. The sharp constant $C_{s,p}$ is given as
\begin{equation}
C_{s,p}^{-1} = (2ps)^{-1} (p - 2)^{\frac{(p-2)(1+s)}{4s}} s^{\frac{p-2}{4}} \left(2(1 + s) - p(1 - s)\right)^{\frac{4s - (p-2)(1+s)}{4s}} \| \nu \|_{L^2}^{p-2},
\label{eq:constant}
\end{equation}
where $\nu$ is a ground state of \eqref{2.1}. Using the arguments given in \cite[page no. 7]{MR3428469} , it is possible to show that there is a constant $C_H > 0$ such that
\begin{equation*}
  {\int_{\mathbb{R}^{2}}|u|^p}\leq C_H \Big({\int_{\mathbb{R}^{2}} \big(|\partial_{x}u|^2} +{\int_{\mathbb{R}^{2}} |(-\Delta)_{y}^{s /2} u|^2 \ }\Big)^{\frac{(p-2)(1+s)}{4s}}\Big({\int_{\mathbb{R}^{2}}|u|^2\Big)}^{\frac{p}{2}-\frac{(p-2)(1+s)}{4s}} \quad ,\forall{~ u\in \mathcal{H}^{1,s}(\R^2)}\\ , 
\end{equation*}
and the sharp constant $C_H$ is
\begin{equation*}
C_H^{-1} = (2ps)^{-1} ((p - 2)(s + 1))^{\frac{(p-2)(1+s)}{4s}} \left(2(1 + s) - p(1 - s)\right)^{\frac{4s - (p-2)(1+s)}{4s}} \| \nu \|_{L^2}^{p-2}.
\label{eq:sharp_constant}
\end{equation*}
For any $(u,v)\in \mathcal{P}$, equation $P(u,v)=0$ suggests that one or more of the following must be true:
\begin{itemize}
\item[(i)] $\displaystyle \frac{1}{3}\Big[{\int_{\mathbb{R}^{2}} \big(|\partial_{x}u|^2 +  |(-\Delta)_{y}^{s /2} u|^2 \big)\,\mathrm{d}x \mathrm{d}y +\int_{\mathbb{R}^{2}} \big(|\partial_{x}v|^2 +  |(-\Delta)_{y}^{s /2} v|^2  \big)\,\mathrm{d}x \mathrm{d}y   } \Big]\leq \frac{(p-2)(1+s)}{2sp}\mu_1\|u\|_p^p$,
\item[(ii)] $\displaystyle \frac{1}{3}\Big[{\int_{\mathbb{R}^{2}} \big(|\partial_{x}u|^2 +  |(-\Delta)_{y}^{s /2} u|^2 \big)\,\mathrm{d}x \mathrm{d}y +\int_{\mathbb{R}^{2}} \big(|\partial_{x}v|^2 +  |(-\Delta)_{y}^{s /2} v|^2  \big)\,\mathrm{d}x \mathrm{d}y   } \Big]\leq \frac{(q-2)(1+s)}{2sq}\mu_2\|v\|_q^q$,
\item[(iii)]$\displaystyle\frac{1}{3}\Big[{\int_{\mathbb{R}^{2}} \big(|\partial_{x}u|^2 +  |(-\Delta)_{y}^{s /2} u|^2 \big)\,\mathrm{d}x \mathrm{d}y +\int_{\mathbb{R}^{2}} \big(|\partial_{x}v|^2 +  |(-\Delta)_{y}^{s /2} v|^2  \big)\,\mathrm{d}x \mathrm{d}y   } \Big]\\
\quad \quad\leq \frac{(r_1+r_2-2)(1+s)}{2s}\beta \int_{\R^2} |u|^{r_1} |v|^{r_2} \mathrm{d}x\mathrm{d}y$.
\end{itemize}
We now analyze each case separately.\\
{\bf Case $(i) :$} Let $u,v\in {D}_a\times{D}_b$ then $\|u\|_2^2\leq a$ and $\|v\|_2^2\leq b$. Now if $(i)$ holds, we have that
\[\begin{split}
\frac{1}{3}&\Big[{\int_{\mathbb{R}^{2}} \big(|\partial_{x}u|^2 +  |(-\Delta)_{y}^{s /2} u|^2 \big)\,\mathrm{d}x \mathrm{d}y +\int_{\mathbb{R}^{2}} \big(|\partial_{x}v|^2 +  |(-\Delta)_{y}^{s /2} v|^2  \big)\,\mathrm{d}x \mathrm{d}y   } \Big] \\
&\leq  \frac{(p-2)(1+s)}{2sp}\mu_1 C_H\Big[{\int_{\mathbb{R}^{2}} \big(|\partial_{x}u|^2 +  |(-\Delta)_{y}^{s /2} u|^2 \big)\,\mathrm{d}x \mathrm{d}y}\Big]^{\frac{(p-2)(1+s)}{4s}}\Big[\int_{\mathbb{R}^{2}}|u|^{2}\mathrm{d}x \mathrm{d}y\Big]^{\frac{p}{2}-\frac{(p-2)(1+s)}{4s}}\\
&\leq \frac{(p-2)(1+s)}{2sp}\mu_1 C_H \left[{\int_{\mathbb{R}^{2}} \big(|\partial_{x}u|^2 +  |(-\Delta)_{y}^{s /2} u|^2 \big)\,\mathrm{d}x \mathrm{d}y +\int_{\mathbb{R}^{2}} \big(|\partial_{x}v|^2 +  |(-\Delta)_{y}^{s /2} v|^2  \big)\,\mathrm{d}x \mathrm{d}y   }\right]^{\frac{(p-2)(1+s)}{4s}} \\
&   \qquad\qquad\times   \Big[{\int_{\mathbb{R}^{2}}\big(|u|^2+|v|^2}\big)\mathrm{d}x \mathrm{d}y\Big]^{\frac{p}{2}-\frac{(p-2)(1+s)}{4s}}.
\end{split}\]
By $\frac{2(1+3s)}{1+s}<p,q,r_1+r_2<2_s=\frac{2(1+s)}{1-s}$, we have that $\frac{(p-2)(1+s)}{4s}>1$ and ${\frac{p}{2}-\frac{(p-2)(1+s)}{4s}}>0$. Hence,
\begin{align}\label{4.3}
{\Big[{\int_{\mathbb{R}^{2}} \big(|\partial_{x}u|^2 +}}&{{|(-\Delta)_{y}^{s /2} u|^2 \big)\,\mathrm{d}x \mathrm{d}y }+{\int_{\mathbb{R}^{2}} \big(|\partial_{x}v|^2 +  |(-\Delta)_{y}^{s /2} v|^2  \big)\,\mathrm{d}x \mathrm{d}y   } \Big]}\nonumber\\
&\geq \left(\frac{3(p-2)(1+s)}{2sp}\mu_1 C_H\right)^{-\frac{4s}{(1+s)(p-2)-4s}} \Big[{\int_{\mathbb{R}^{2}}\big(|u|^2+|v|^2\big)\mathrm{d}x \mathrm{d}y}\Big]^{-\frac{2sp-(p-2)(1+s)}{(1+s)(p-2)-4s}}.\nonumber\\
&\geq \left(\frac{3(p-2)(1+s)}{2sp}\mu_1 C_H\right)^{-\frac{4s}{(1+s)(p-2)-4s}} \Big[\frac{1}{a+b}\Big]^{\frac{2sp-(p-2)(1+s)}{(1+s)(p-2)-4s}}.
\end{align}
{\bf Case $(ii)$:} Similarly, if $(ii)$ holds, we have
\begin{align}\label{4.4}
\Big[{\int_{\mathbb{R}^{2}} \big(|\partial_{x}u|^2 +} & {|(-\Delta)_{y}^{s /2} u|^2 \big)\,\mathrm{d}x \mathrm{d}y} +{\int_{\mathbb{R}^{2}} \big(|\partial_{x}v|^2 +  |(-\Delta)_{y}^{s /2} v|^2  \big)\,\mathrm{d}x \mathrm{d}y   } \Big]\nonumber\\
&\geq \left(\frac{3(q-2)(1+s)}{2sq}\mu_2 C_H\right)^{-\frac{4s}{(1+s)(q-2)-4s}} \Big[{\int_{\mathbb{R}^{2}}\big(|u|^2+|v|^2\big)\mathrm{d}x \mathrm{d}y}\Big]^{-\frac{2sq-(q-2)(1+s)}{(1+s)(q-2)-4s}}.\nonumber\\
&\geq \left(\frac{3(q-2)(1+s)}{2sq}\mu_2 C_H\right)^{-\frac{4s}{(1+s)(q-2)-4s}} \Big[\frac{1}{a+b}\Big]^{\frac{2sq-(q-2)(1+s)}{(1+s)(q-2)-4s}}.
\end{align}
{\bf Case $(iii)$:}
In the event that $(iii)$ holds, H\"older inequality shows that
\begin{align}\label{4.5}
    \Big[{\int_{\mathbb{R}^{2}} \big(|\partial_{x}u|^2} &{+  |(-\Delta)_{y}^{s /2} u|^2 \big)\,\mathrm{d}x \mathrm{d}y }+{\int_{\mathbb{R}^{2}} \big(|\partial_{x}v|^2 +  |(-\Delta)_{y}^{s /2} v|^2  \big)\,\mathrm{d}x \mathrm{d}y   } \Big]\nonumber\\
    &\geq \left(\frac{3(r_1+r_2-2)(1+s)}{2s}\beta C_H\right)^{-\frac{4s}{(1+s)(r_1+r_2-2)-4s}}   
 \Big[{\int_{\mathbb{R}^{2}}\big(|u|^2+|v|^2\big)\mathrm{d}x \mathrm{d}y}\Big]^{-\frac{2s(r_1+r_2)-(r_1+r_2-2)(1+s)}{(1+s)(r_1+r_2-2)-4s}} .\nonumber\\
  &\geq \left(\frac{3(r_1+r_2-2)(1+s)}{2s}\beta C_H\right)^{-\frac{4s}{(1+s)(r_1+r_2-2)-4s}}   
 \Big[\frac{1}{a+b}\Big]^{\frac{2s(r_1+r_2)-(r_1+r_2-2)(1+s)}{(1+s)(r_1+r_2-2)-4s}}.
\end{align}
In each of the above cases from \eqref{4.3}-\eqref{4.5}, we can find some $\delta_{(a,b)}>0$ such that
\begin{equation*}
\Big[{\int_{\mathbb{R}^{2}} \big(|\partial_{x}u|^2 +  |(-\Delta)_{y}^{s /2} u|^2 \big)\,\mathrm{d}x \mathrm{d}y +\int_{\mathbb{R}^{2}} \big(|\partial_{x}v|^2 +  |(-\Delta)_{y}^{s /2} v|^2  \big)\,\mathrm{d}x \mathrm{d}y   }\Big]\geq \delta_{(a,b)}, \quad \mbox{for all } (u,v)\in \mathcal{P}_{(a,b)}.
\end{equation*}
This finishes the proof.
\end{proof}
\begin{corollary}\label{corollary4.7}
For any $(a,b)\in \R^+\times \R^+$,
$$C_{(a,b)}=\inf_{(0,0)\neq (u,v)\in D_a\times D_b}\max_{t>0} J(u_t,v_t)>0.$$
\end{corollary}
\begin{proof}
For any $(u,v)\in \mathcal{P}_{(a,b)}$, by Lemma \ref{lemma4.4} and Lemma \ref{lemma4.6}, we have that
$$J(u,v)\geq C_0\Big[{\int_{\mathbb{R}^{2}} \big(|\partial_{x}u|^2 +  |(-\Delta)_{y}^{s /2} u|^2 \big)\,\mathrm{d}x \mathrm{d}y +\int_{\mathbb{R}^{2}} \big(|\partial_{x}v|^2 +  |(-\Delta)_{y}^{s /2} v|^2  \big)\,\mathrm{d}x \mathrm{d}y   }\Big]\geq C_0\delta_{(a,b)}>0.$$
Hence, $C_{(a,b)}$ is well defined and $C_{(a,b)}\geq C_0\delta_{(a,b)}>0$. Furthermore, by {Corollary \ref{corollary4.3}}, one can see that
 $$\inf_{\mathcal{P}_{(a,b)}}J(u,v)=\inf_{(0,0)\neq (u,v)\in D_a\times D_b}\max_{t>0} J(u_t,v_t).$$
 This finishes the proof.
\end{proof}
In conclusion of this section, we demonstrate that every critical point of $J$ restricted to $\mathcal{P}_{(a,b)}$ is a critical point of $J$ restricted to $D_a \times D_b$. This suggests, specifically, that a ground state for problem \eqref{eq:20220902-maine1} is a minimizer of $C_{(a,b)}$.
\begin{lemma}\label{lemma4.8}
For any $(u,v)\in \mathcal{P}$ which is a critical point of $J\big|_{\mathcal{P}_{(a,b)}}$ satisfying $\phi''_{(u,v)}(1)\neq 0$, then there exist  $\lambda_1,\lambda_2\in \R$ such that
\begin{equation*}
J'(u,v)+\lambda_1(u,0)+\lambda_2(0,v)=0.
\end{equation*}
\end{lemma}
\begin{proof}
First of all, it is evident that there is $\lambda_1,\lambda_2,\mu\in \R$ such that
\begin{equation}\label{4.7}
J'(u,v)+\lambda_1(u,0)+\lambda_2(0,v)+\mu P'(u,v)=0~\hbox{in $\mathbb{D}^{-1}$, the dual space of}~\mathbb{D}.
\end{equation}
Thus, all we have to do is demonstrate that $\mu=0$. Let us demonstrate the following functional using \eqref{4.7} as
\begin{equation*}
\psi(u,v):=J(u,v)+\frac{1}{2}\int_{\mathbb{R}^{2}}\lambda_1|u|^2+\frac{1}{2}\int_{\mathbb{R}^{2}}\lambda_2|v|^2+\mu P(u,v).
\end{equation*}
Given $(u,v)$ solves \eqref{4.7}, it follows that $t=1$ is a critical point of $\tilde{\phi}_{(u,v)}:\R^+\rightarrow \R$ defined by
\begin{align*}
\tilde{\phi}_{(u,v)}(t) := \psi(u_t,v_t)
=\phi_{(u,v)}(t)+\frac{1}{2}\int_{\mathbb{R}^{2}}\lambda_1|u|^2+\lambda_2|v|^2+\mu t\phi'_{(u,v)}(t).
\end{align*}
Now, we have
\begin{equation*}
\tilde{\phi}'_{(u,v)}(t) = \frac{d}{dt}\psi(u_t,v_t)=(1+\mu)\phi'_{(u,v)}(t)+\mu \phi''_{(u,v)}(t).
\end{equation*}
We obtain that $\mu \phi''_{(u,v)}(1)=0$ by substituting $t=1$ and $\phi'_{(u,v)}(1)=P(u,v)=0$. This says that $\phi''_{(u,v)}(1)\neq 0$ implies $\mu=0$.
\end{proof}
\begin{lemma}\label{lemma4.9}
For any $(u,v)\in \mathcal{P}$, we have $\phi''_{(u,v)}(1)<0$.
\end{lemma}
\begin{proof}
Through direct calculation, we obtain that
\begin{align*}
\phi''_{(u,v)}(1)=& {s(2s-1)}\Big[{\int_{\mathbb{R}^{2}} \big(|\partial_{x}u|^2 +  |(-\Delta)_{y}^{s /2} u|^2 \big)\,\mathrm{d}x \mathrm{d}y +\int_{\mathbb{R}^{2}} \big(|\partial_{x}v|^2 +  |(-\Delta)_{y}^{s /2} v|^2  \big)\,\mathrm{d}x \mathrm{d}y   } \Big]\\
&
-\frac{(p-2)(1+s)}{2p}\frac{(p-2)(1+s)-2}{2}\mu_1  
 \int_{\mathbb{R}^2}  |u|^p  \mathrm{d}x \mathrm{d}y                         \nonumber\\
&-\frac{(q-2)(1+s)}{2q}\frac{(q-2)(1+s)-2}{2}\mu_2  
 \int_{\mathbb{R}^2}  |u|^q  \mathrm{d}x \mathrm{d}y \nonumber\\
&-\frac{(r_1+r_2-2)(1+s)}{2}\frac{(r_1+r_2-2)(1+s)-2}{2}\beta \Big( \int_{\R^2} |u|^{r_1} |v|^{r_2} \mathrm{d}x  \mathrm{d}y\Big)         
\end{align*}
In contrast, we have for any $(u,v)\in \mathcal{P}$,
\begin{align*}
 &s\int_{\mathbb{R}^{2}} \big(|\partial_{x}u|^2 +  |(-\Delta)_{y}^{s /2} u|^2  \big)\,\mathrm{d}x \mathrm{d}y 
  + s\int_{\mathbb{R}^{2}} \big(|\partial_{x}v|^2 +  |(-\Delta)_{y}^{s /2} v|^2 \big)\,\mathrm{d}x \mathrm{d}y \\
  &\quad= \frac{(1+s)(p-2)}{2p}\mu _1\int_{\mathbb{R}^{2}} |u|^{p}\, \mathrm{d}x \mathrm{d}y    +\frac{(1+s)(q-2)}{2q}\mu _2\int_{\mathbb{R}^{2}} |v|^{q}\, \mathrm{d}x \mathrm{d}y  +{\frac{(1+s)(r_1+r_2-2)}{2}} \beta \int_{\mathbb{R}^{2}}|u|^{r_1}|v|^{r_2}\, \mathrm{d}x \mathrm{d}y.
\end{align*}
Using this, we can rewrite
\begin{align*}
\phi''_{(u,v)}(1)=&-\left[(p-2)(1+s)-4s\right] \frac{(p-2)(1+s)}{4p}\mu_1\int_{\mathbb{R}^2}|u|^p\mathrm{d}x \mathrm{d}y\nonumber\\
&\quad-\left[(q-2)(1+s)-4s\right]\frac{(q-2)(1+s)}{4q}\mu_2\int_{\mathbb{R}^2}|u|^q\mathrm{d}x \mathrm{d}y \nonumber\\
&\quad \quad-\left[(r_1+r_2-2)(1+s)-4s\right]\frac{(r_1+r_2-2)(1+s)}{4}\beta\int_{\R^2}|u|^{r_1}|v|^{r_2}\mathrm{d}x  \mathrm{d}y.
\end{align*}
Imposing the conditions  $\frac{2(1+3s)}{1+s}<p,q,r_1+r_2<2_s=\frac{2(1+s)}{1-s}$ and $(u,v)\neq (0,0)$ in above, we conclude  that $\phi''_{(u,v)}(1)<0$.
\end{proof}
Assume that $(u^*, v^*)$ and $(u^\#, v^\#)$ is the Steiner symmetrization of $u$ and $v$ for any $(u,v)\in \mathcal{P}_{(a,b)}$ w.r.t. $x$ and $y$, respectively. Notice that $(u^*)^\#=(u^\#)^*$ is a function symmetric to both $x$ and $y$. For a detailed study on Steiner symmetrization, we refer \cite{shibata2017new}. Following is an immediate outcome.
\begin{lemma}\label{lemma5.1}
Let  $\frac{2(1+3s)}{1+s}<p,q,r_1+r_2<2_s=\frac{2(1+s)}{1-s}. $
For any $(u,v)\in \mathcal{P}_{(a,b)}$, there exists a unique $t_0=t_{(u^*,v^*)}\in (0,1]$ such that $((u_t)^*,  (v_t)^*)\in \mathcal{P}_{(a,b)}$ and  $J((u_t)^*,  (v_t)^*)\leq J(u,v) $ $w.r.t.$ x. A similar result holds for the symmetrization $w.r.t.$ y.
\end{lemma}
\begin{proof}
For any $(u,v)\in \mathcal{P}_{(a,b)}$, we have $(u,v)\neq (0,0)$. So by $\|u^*\|_2^2=\|u\|_2^2$ and $\|v^*\|_2^2=\|v\|_2^2$, we see that $(u^*,v^*)\in (D_a\times D_b)\backslash \{(0,0)\}$. Then by Corollary \ref{corollary4.3}, there exists a unique
$t_0=t_{(u^*,v^*)}>0$ such that $( u^*_t, v^*_t)\in \mathcal{P}_{(a,b)}$.

By applying the result from \cite[Proposition 3.1]{gou2023solitary}, we obtain that $P(u^*, v^*) \leq P(u, v) = 0$. Consequently, by Corollary \ref{corollary4.3}, it follows that $t_{(u^*, v^*)} \leq 1$. Moreover, as noted, we also have $J(u^*, v^*) \leq J(u, v)$ for every $(u, v) \neq (0,0)$, again using the result from \cite[Proposition 3.1]{gou2023solitary}. Finally, combining these observations with the fact that
\[
(u^*)_k = (u_k)^*, \quad \text{for all } k \in \mathbb{R}^+,
\]
where $u_k(x, y) = k^{\frac{1+s}{2}} u(k^s x, k y)$, We finally obtain
\begin{align*}
\max_{k>0}J((u^*)_k, (v^*)_k)=J((u^*)_t,  (v^*)_t)=J(( u_t)^*, (v_t)^*)
\leq J( u_t, v_t )
\leq\max_{k>0}J( u_k,  v_k)
=J(u,v).
\end{align*}
This finishes the proof.
\end{proof}
\begin{remark}\label{remark5.2}
Specifically, Lemma \ref{lemma5.1} suggests that a pair $(u,v) \in \mathbb{D}$ of Steiner symmetric functions would attain $C_{(a,b)}$.
\end{remark}
\section{Existence of a special (PS)-sequence }\label{sec:Palais_Smale}
\renewcommand{\theequation}{4.\arabic{equation}}
In this segment, we establish the existence of a special Palais-Smale sequence for $J$ restricted to $D_a\times D_b$ at level $C_{(a,b)}$. 
For this, we first define the class of paths
\begin{equation*}
\Gamma_{(a,b)} = \{g \in C([0,1], D_a\times D_b) : g(0) \in L_{(a,b)} \,\,  \mbox{and} \,\, P(g(1)) <0 \}
\end{equation*}
where $$L_{(a,b)} := \left\{ (u,v) \in \mathbb{D} : \Big[{\int_{\mathbb{R}^{2}} \big(|\partial_{x}u|^2 +  |(-\Delta)_{y}^{s /2} u|^2 \big)\,\mathrm{d}x \mathrm{d}y +\int_{\mathbb{R}^{2}} \big(|\partial_{x}v|^2 +  |(-\Delta)_{y}^{s /2} v|^2  \big)\,\mathrm{d}x \mathrm{d}y   } \Big]\leq \frac{1}{2}\delta_{(a,b)}\right\}.$$
Observe that $\emptyset \neq \Gamma_{(a,b)}$. In fact, $L_{(a,b)} \neq \emptyset$, and we know that the set where $P<0$ is non-void from Corollary \ref{corollary4.3} since the map $(u,v)\to (u_t,v_t)$ is the continuous and one-one map. 
Subsequently, we define
\begin{equation*}
\gamma_{(a,b)} = \inf_{g \in \Gamma_{(a,b)}} \max_{t \in [0,1]} J(g(t)).
\end{equation*}
and we claim that $$\gamma_{(a,b)} = C_{(a,b)}.$$We can see that $P >0$ on $L_{(a,b)}$ as a consequences of combining Corollary \ref{corollary4.3} with Lemma \ref{lemma4.6}. Consequently, due to continuity of $g$, there exists a $t \in [0,1]$ such that $P(g(t)) =0$ for any $g \in \Gamma_{(a,b)}$. It suggests that $C_{(a,b)} \leq \gamma_{(a,b)}$. On the other hand, Corollary \ref{corollary4.3} yields the converse inequality which is $\gamma_{(a,b)} \leq C_{(a,b)}$.  
The equality $\gamma_{(a,b)} = C_{(a,b)}$ leads us to the conclusion that 
$$ \inf_{(u,v) \in \mathcal{P}_{(a,b)} } J(u,v) \geq \gamma_{(a,b)}.$$
We note that we can now apply \cite[Theorem 4.1]{MR1251958} which asserts that for any collection of pathways $\{ g_n \} \subset \Gamma_{(a,b)}$ such that  as  $n \rightarrow \infty$,
\begin{equation}\label{mimpaths}
\max_{[0,1]} J(g_n(t)) \rightarrow C_{(a,b)},
\end{equation}
 there exists a sequence $\{(u_n,v_n) \} \subset \mathbb{D}$ such that as $n \rightarrow \infty$,
\begin{itemize}
\item[(i)] $J(u_n,v_n) \rightarrow C_{(a,b)}$,
\item[(ii)] $ J'(u_n,v_n)\Big|_{D_a \times D_b}\rightarrow 0$,
\item[(iii)] $dist((u_n,v_n), \mathcal{P}_{(a,b)}) \rightarrow 0$, 
\item[(iv)] $dist((u_n,v_n), g_n[0,1]) \rightarrow 0.$
\end{itemize}
Due to Lemma \ref{lemma5.1}, let $((u_n^*)^\#, (v_n^*)^\#) \subset \mathcal{P}_{(a,b)}$ be a minimizing sequence for $C_{(a,b)}$. Now, let $\{g_n\} \subset \Gamma_{(a,b)}$ be the sequence of pathways as follows:
$$g_n(t) = (  (u_n^*)_{k_1^n + t k_2^n}^\#, (v_n^*)_{k_1^n + t k_2^n}^\# )$$
where $(u_n^*)_{k_1^n + t k_2^n}^\#=(k_1^n+tk_2^n)^{\frac{1+s}{2}}(u_n^*)^\#((k_1^n+tk_2^n)^sx,(k_1^n+tk_2^n)y)$, $k_1 >0$ sufficiently tiny and $k_2 >0$ sufficiently large. Hence $g_n(0)\in L_{(a,b)}$ and $P(g_n(1))<0 $ which states that $g_n \in \Gamma_{(a,b)}$. Now by using the Corollary \ref{corollary4.7} we can say that sequence $\{g_n\}$ satisfies \eqref{mimpaths} i.e. 
\begin{equation*}
\max_{[0,1]} J(g_n(t)) \rightarrow C_{(a,b)}.
\end{equation*}
 \smallskip
By the help of (i)-(ii) we can see that $\{(u_n,v_n) \} \subset \mathbb{D}$ is a Palais-Smale sequence for $J$ restricted to $D_a \times D_b$ at the level $C_{(a,b)}$. The property (iii), combined with the fact that $J$ is coercive on $\mathcal{P}_{(a,b)}$(see
Corollary \ref{corollary4.5}) and that $J'$ and $P'$ take bounded sets into bounded set implies that $\{(u_n,v_n) \} \subset \mathbb{D}$ is bounded and that $P(u_n,v_n) \rightarrow 0$.  Now by compactness of the 
embedding, for each $n \in \mathbb{N}$, there exists $k_n \in ({k_1^n}, {k_1^n}+{k_2^n})$  such that
\begin{equation}\label{distance_to_paths}
\text{dist} \Big((u_n,v_n), g_n([0,1])\Big) = \Big|\Big| (u_n,v_n) - ( (u_n^*)_{{k_n}}^\# , (v_n^*)_{{k_n}}^\#) \Big|\Big|.
\end{equation}
Up to a subsequence we can assume that $((u_n^*)_{k_n}^\#, (v_n^*)_{k_n}^\#) \rightharpoonup (u,v)$ weakly in $\mathbb{D}$ and
$((u_n^*)_{k_n}^\#, (v_n^*)_{k_n}^\#) \rightarrow (u,v) $
 strongly in $L^{\eta}(\R^2) \times L^{\eta}(\R^2)$ for all $2 < \eta < 2_{s}$. The fact that the functions are radial leads to this strong convergence. The Palais-Smale sequence $\{(u_n,v_n) \} \subset \mathbb{D}$ exhibits convergence properties as a result of \eqref{distance_to_paths}.
 
We have, after summarizing, a bounded sequence $\{(u_n,v_n) \} \subset D_a\times D_b$, such that  
\begin{itemize}
\item[(i)] $J(u_n,v_n)\rightarrow C_{(a,b)}$,
\item[(ii)] $ J'(u_n,v_n) + \lambda_{1,n}(u_n,0) + \lambda_{2,n}(0,v_n) \to 0, \quad  \mbox{in } \mathbb{D}^{-1}$ for some real sequences $\{\lambda_{1,n}\}$ and $\{\lambda_{2,n}\}$,
\item[(iii)] $ P(u_n,v_n) \rightarrow 0$,
\item[(iv)] $(u_n,v_n) \rightharpoonup (u,v) $   weakly in  $\mathbb{D}$ and  $(u_n,v_n) \rightarrow (u,v) $ strongly in  $L^{\eta}(\R^2) \times L^{\eta}(\R^2)$ for all $2 < \eta < 2_{s}$. Additionally $u \geq 0$ and $v\geq 0$.
\end{itemize}

\subsection{Estimation of $C_{(a,b)}$:}\label{sec:estimation}
\renewcommand{\theequation}{6.\arabic{equation}}
With the help of $z_{p,\mu,a}$  defined in \eqref{2.5} and $\beta_{p,\mu,a,r}$ defined in \eqref{def:-best-const}, we can observe that
\begin{equation*}
\beta_{p,\mu,a,r}=\frac{1}{2}\lambda^{\frac{p-2-r}{p-2}} \mu^{\frac{r}{p-2}}
\inf_{h\in \mathcal{H}^{1,s}(\R^2)\backslash\{0\}}\frac{\Big[{\int_{\mathbb{R}^{2}} \big(|\partial_{x}h|^2 +  |(-\Delta)_{y}^{s /2} h|^2 \big)\,\mathrm{d}x \mathrm{d}y }\Big]}{\int_{\R^2}u_p^r h^2 \mathrm{d}x\mathrm{d}y}.
\end{equation*}
 Now substituting the value of $\lambda$ (see \eqref{2.4}), we obtain that
\begin{equation*}
\label{eq:20221212-be1}
\beta_{p,\mu,a,r}=\frac{1}{2}\mu^{\frac{r(1+s)-4s}{(1+s)(p-2)-4s}} \,
\|u\|_{2}^{\frac{4s(p-2-r)}{(1+s)(p-2)-4s}} a^{-\frac{2s(p-2-r)}{(1+s)(p-2)-4s}} \inf_{h\in H^{1,s}(\R^2)\backslash\{0\}}\frac{\Big[{\int_{\mathbb{R}^{2}} \big(|\partial_{x}h|^2 +  |(-\Delta)_{y}^{s /2} h|^2 \big)\,\mathrm{d}x \mathrm{d}y }\Big]} {\int_{\R^2}u_p^r h^2 \mathrm{d}x\mathrm{d}y}.
\end{equation*}
\begin{lemma}\label{lemma7.1}
For $p\in (2,+\infty), \mu, a, r>0$, it holds that
$\beta_{p,\mu,a,r}=0$.
\end{lemma}
\begin{proof}
 We remark that the embedding $\mathcal{H}^{1,s}(\R^2)\hookrightarrow L^\infty(\R^2)$ does not hold true, so there exists a sequence $\{\psi_n\}\subset \mathcal{H}^{1,s}(\R^2)$ with $\|\psi_n\|_{\mathcal{H}^{1,s}}=1$, but $\|\psi_n\|_\infty\rightarrow \infty$ as $n\rightarrow \infty$. In particular, by the property of rearrangement, without loss of generality, we can assume that $\psi_n=(\psi_n^*)^\#$. So we have that
$$\psi_n(x)=\psi_n(|x|)~\hbox{and}~\psi_n(0)=\|\psi_n\|_\infty.$$
Now, we set
$$h_n(x)=\begin{cases}
\psi_n(0)\quad &\hbox{if}~|x|\leq n,\\
\psi_n(|x|-n)\quad &\hbox{if}~|x|\geq n.
\end{cases}$$
Then $\|\partial_x h_n\|_{2}^{2}+\|(-\Delta)_{y}^{s /2} h_n\|_{2}^{2}=\|\partial_x \psi_n\|_{2}^{2}+\|(-\Delta)_{y}^{s /2} \psi_n\|_{2}^{2}\leq \|\psi_n\|_{\mathcal{H}^{1,s}}^{2}=1$ and
\begin{align*}
&\int_{\R^2}u_p^r(x)h_n^2(x) \mathrm{d}x\mathrm{d}y\geq \int_{|x|\leq n} \psi_n^2(0) u_p^r(x) \mathrm{d}x
=\|\psi_n\|_\infty^2 \int_{|x|\leq n} u_p^r(x)\mathrm{d}x\mathrm{d}y
\rightarrow \infty~\hbox{as}~n\rightarrow \infty.
\end{align*}
This implies
$$\inf_{0\neq h\in \mathcal{H}^{1,s}(\R^2)}\frac{\Big[{\int_{\mathbb{R}^{2}} \big(|\partial_{x}h|^2 +  |(-\Delta)_{y}^{s /2} h|^2 \big)\,\mathrm{d}x \mathrm{d}y }\Big]} {\int_{\R^2}u_p^r h^2 \mathrm{d}x \mathrm{d}y}=0,$$
that is $\beta_{p,\mu,a,r}=0$\\
This finishes the proof.
\end{proof}
\begin{lemma}  
\label{lemma7.2}
\begin{itemize}
\item[(i)]If $1<r_1<2$, then $C_{(a,b)}<m_{q,\mu_2,b}$. If $r_1=2$, then  $C_{(a,b)}<m_{q,\mu_2,b}$ provided $\beta>\beta_{q,\mu_2,b,r_2}$.
\item[(ii)]If $1<r_2<2$, then $C_{(a,b)}<m_{p,\mu_1,a}$. If $r_2=2$, then $C_{(a,b)}<m_{p,\mu_1,a}$ provided $\beta>\beta_{p,\mu_1,a,r_1}$.
\end{itemize}
\end{lemma}
\begin{proof}
We will show the proof of only result (ii), the proof of (i) being identical in nature. We write $z:=z_{p,\mu_1,a}$ in order to keep things simple. Let for any $h\in \mathcal{H}^{1,s}(\R^2)$ with $\|h\|_2^2=1$, we have that $(z, lh)\in D_a\times D_b$  provided $|l|<\sqrt{b}$.
For any $l$, there exists a unique $t=t(l)>0$ such that
$( z_t,  lh_t)\in \mathcal{P}_{(a,b)}$, where $t=t(l)$ is determined by
\begin{align*}
{\int_{\mathbb{R}^{2}}}&{ \big(|\partial_{x}z|^2 +  |(-\Delta)_{y}^{s /2} z|^2}\big)\mathrm{d}x \mathrm{d}y +l^2{\int_{\mathbb{R}^{2}} \big(|\partial_{x}h|^2} +  {|(-\Delta)_{y}^{s /2} h|^2}\big)\mathrm{d}x \mathrm{d}y\nonumber\\
&=\frac{(p-2)(1+s)}{2ps}\mu_1\|z\|_p^p \, t^{\frac{(p-2)(1+s)-4s}{2}}
+\frac{(q-2)(1+s)}{2qs}\mu_2\|h\|_q^q \,  l^q \, t^{\frac{(q-2)(1+s)-4s}{2}}\nonumber\\
&\quad+\frac{(r_1+r_2-2)(1+s)}{2s}\beta \Big(\int_{\R^2}|z|^{r_1}|h|^{r_2}\mathrm{d}x\mathrm{d}y\Big) l^{r_2} \,  t^{\frac{(r_1+r_2-2)(1+s)-4s}{2}}
\end{align*}
From the Implicit Function Theorem, we know that $t(l)\in C^1$ locally around $l=0$.
 Taking the derivative with respect to $l$ on both sides we obtain that
 \begin{align*}
 t'(l)=\frac{P_h(l)}{Q_h(l)},
 \end{align*}
 where
 \begin{align*}
 P_h(l):=&2l{\int_{\mathbb{R}^{2}} \big(\big|\partial_{x}h|^2 +  |(-\Delta)_{y}^{s /2} h|^2}\big)\mathrm{d}x \mathrm{d}y  -\frac{(q-2)(1+s)}{2s}\mu_2 \, \|h\|_q^q \, |l|^{q-2}l \, t^{\frac{(q-2)(1+s)-4s}{2}}\nonumber\\
 &-\frac{(r_1+r_2-2)(1+s)}{2s}\beta \, r_2 \Big(\int_{\R^2}|z|^{r_1} |h|^{r_2}\mathrm{d}x\mathrm{d}y \Big) |l|^{r_2-2}l \, t^{\frac{(r_1+r_2-2)(1+s)-4s}{2}}\nonumber\\
 \end{align*}
 and
 \begin{align*}
 Q_h(l):=&\frac{(p-2)(1+s)-4s}{2}\frac{(p-2)(1+s)}{2ps}\mu_1 \, \|z\|_p^p \,  t^{\frac{(p-2)(1+s)-4s-2}{2}}\nonumber\\
 &+\frac{(q-2)(1+s)-4s}{2}\frac{(q-2)(1+s)}{2qs}\mu_2 \, \|h\|_q^q \,  l^q \, t^{\frac{(q-2)(1+s)-4s-2}{2}}\nonumber\\
 &+\frac{(r_1+r_2-2)(1+s)-4s}{2}\frac{(r_1+r_2-2)(1+s)}{2s}\beta \Big(\int_{\R^2}|z|^{r_1}|h|^{r_2}\mathrm{d}x \mathrm{d}y\Big) l^{r_2} \, t^{\frac{(r_1+r_2-2)(1+s)-4s-2}{2s}}\nonumber\\
 \end{align*}
  {\bf Case $(i) ~1<r_2<2:$}  In such a case, for $|l|$ small,
 $$P_h(l)=-\frac{(r_1+r_2-2)(1+s)}{2s}\beta r_2 \Big(\int_{\R^2}|z|^{r_1} |h|^{r_2}\mathrm{d}x \mathrm{d}y  \Big) |l|^{r_2-2}l (1+o(1))$$
 and
 $$Q_h(l)=\frac{(p-2)(1+s)-4s}{2} \, \frac{(p-2)(1+s)}{2ps}\mu_1\|z\|_p^p (1+o(1)).$$
 So,
 \begin{equation*}
 t'(l)=-M_h r_2 |l|^{r_2-2}l (1+o(1)),
 \end{equation*}
 where
 \begin{equation*}
 M_h:=\frac{{(r_1+r_2-2)(1+s)}{\beta \int_{\R^2}|z|^{r_1} |h|^{r_2}\mathrm{d}x\mathrm{d}y}}{({(p-2)(1+s)-4s})s{\int_{\mathbb{R}^{2}} \big(|\partial_{x}z|^2 +  |(-\Delta)_{y}^{s /2} z|^2}\big)\mathrm{d}x \mathrm{d}y}.
 \end{equation*}
 Then,
 \begin{equation*}
 t(l)=1-M_h|l|^{r_2} (1+o(1))
 \end{equation*}
 and for any $\tau>0$,
 \begin{equation*}
 t(l)^\tau=1-\tau M_h |l|^{r_2} (1+o(1)).
 \end{equation*}
Thus, we have
 \begin{align*}
 J( z_{t(l)},  lh_{t(l)})-J(z,0)
 &=\frac{1}{2}(t^{2s}(l)-1){\int_{\mathbb{R}^{2}} \big(\big|\partial_{x}z|^2 +  |(-\Delta)_{y}^{s /2} z|^2}\big)\mathrm{d}x \mathrm{d}y \, -\frac{\mu_1}{p} \, \|z\|_p^p \, (t(l)^{\frac{(p-2)(1+s)}{2}}-1)\\
 &+\frac{1}{2}l^2 \, t^{2s}(l){\int_{\mathbb{R}^{2}} \big(\big|\partial_{x}h|^2 +  |(-\Delta)_{y}^{s /2} h|^2}\big)\mathrm{d}x \mathrm{d}y \,   -\frac{\mu_2}{q} \, \, \|h\|_q^q \, l^q \, t(l)^{\frac{(q-2)(1+s)}{2}}\\
 &-\beta \Big(\int_{\R^2}|z|^{r_1}|h|^{r_2}\mathrm{d}x\mathrm{d}y\Big) |l|^{r_2} \, t(l)^{\frac{(r_1+r_2-2)(1+s)}{2}}\\
 =&-\Big({s}{\int_{\mathbb{R}^{2}} \big(\big|\partial_{x}z|^2 +  |(-\Delta)_{y}^{s /2} z|^2}\big)\mathrm{d}x \mathrm{d}y \Big) M_h \, |l|^{r_2} o(1) -\beta \Big(\int_{\R^2}|z|^{r_1}|h|^{r_2}\mathrm{d}x \mathrm{d}y\Big) |l|^{r_2}(1+o(1))\\
 =&-\beta \Big(\int_{\R^2}|z|^{r_1}|h|^{r_2}\mathrm{d}x \mathrm{d}y\Big) |l|^{r_2}(1+o(1)).
 \end{align*}
Hence, if $1<r_2<2$, by taking $l$ close to $0$, we obtain that
\begin{equation*}
C_{(a,b)}\leq J( z_{t(l)},  lh_{t(l)})<J(z,0)=m_{p,\mu_1,a}.
\end{equation*}

{\bf Case $(ii)~ r_2=2$}: In such a case, we have, for $l>0$ small,
\begin{align*}
 P_h(l)=\Big[2{\int_{\mathbb{R}^{2}} \big(\big|\partial_{x}h|^2} +  {|(-\Delta)_{y}^{s /2} h|^2}\big)\mathrm{d}x \mathrm{d}y-\frac{(1+s)}{s}\beta r_1 \int_{\R^2}|z|^{r_1} |h|^{2}\mathrm{d}x \mathrm{d}y 
 \Big] l (1+o(1))
 \end{align*}
 and
\begin{equation*}
 t'(l)=\bar{M}_h l(1+o(1))
\end{equation*}
where
\begin{equation*}
\bar{M}_h:=\frac{2s{\int_{\mathbb{R}^{2}} \big(|\partial_{x}h|^2} +  {|(-\Delta)_{y}^{s /2} h|^2}\big)\mathrm{d}x \mathrm{d}y-(1+s)\beta r_1 \int_{\R^2}|z|^{r_1} |h|^{2}\mathrm{d}x \mathrm{d}y
 }{{\frac{(p-2)(1+s)-4s}{2}s\int_{\mathbb{R}^{2}} \big(|\partial_{x}z|^2} +  {|(-\Delta)_{y}^{s /2} z|^2}\big)\mathrm{d}x \mathrm{d}y}.
\end{equation*}
Then
$$t(l)=1+\frac{1}{2}\bar{M}_h \,l^2(1+o(1))$$
and for any $\tau>0$,
$$t(l)^\tau=1+\frac{\tau}{2}\bar{M}_h \,l^2(1+o(1)).$$
Thus, using the same reasoning as in the $1<r_2<2$ Case, we arrive at the conclusion that
\begin{align*}
J( z_{t(l)},  lh_{t(l)})-J(z,0)
 =&\frac{1}{2}(t^{2s}(l)-1){\int_{\mathbb{R}^{2}} \big(|\partial_{x}z|^2} +  {|(-\Delta)_{y}^{s /2} z|^2}\big)\mathrm{d}x \mathrm{d}y \,  -\frac{ \mu_1}{p}\,\|z\|_p^p(t(l)^{\frac{(p-2)(1+s)}{2}}-1) \\
 &+\frac{1}{2}l^2 \, t(l)^{2s} \int_{\mathbb{R}^{2}} \big(|\partial_{x}h|^2 +  {|(-\Delta)_{y}^{s /2} h|^2}\big)\mathrm{d}x \mathrm{d}y \, -\frac{\mu_2}{q}\,  \|h\|_q^q \, l^q \, t(l)^{\frac{(q-2)(1+s)}{2}}\\
 &-{\beta} \Big( \int_{\R^2}|z|^{r_1}|h|^{2}\mathrm{d}x \mathrm{d}y\Big) |l|^{2} t(l)^{\frac{r_1(1+s)}{2}}\\
 =&\left[\frac{1}{2}\int_{\mathbb{R}^{2}} \big(|\partial_{x}h|^2 +  {|(-\Delta)_{y}^{s /2} h|^2}\big)\mathrm{d}x \mathrm{d}y -\beta  \int_{\R^2}|z|^{r_1}|h|^{2}\mathrm{d}x\mathrm{d}y \right]l^2(1+o(1)).
\end{align*}
Then for any $\beta>\beta_{p,\mu_1,a,r_1}$, there exists some $h\in \mathcal{H}^{1,s}(\R^2)$ such that
$$\frac{1}{2}\int_{\mathbb{R}^{2}} \big(|\partial_{x}h|^2 +  {|(-\Delta)_{y}^{s /2} h|^2}\big)\mathrm{d}x \mathrm{d}y  -\beta \int_{\R^2}|z|^{r_1}|h|^{2}\mathrm{d}x \mathrm{d}y<0,$$
and for such a $h$ and $|l|$ small, we have
$J(z_{t(l)},  lh_{t(l)})-J(z,0)<0$. Hence,
$$C_{(a,b)}<m_{p,\mu_1,a}.$$
This finishes the proof.
\end{proof}

\section{Proof of Main Theorem }\label{sec:PS-sequence-sec-5} 
\renewcommand{\theequation}{5.\arabic{equation}}

The following outcome will directly lead to Theorem \ref{th:main-t1}.
\begin{theorem}\label{theorem8.1}
Let  $\frac{2(1+3s)}{1+s}<p,q,r_1+r_2<2_s=\frac{2(1+s)}{1-s}$ and $\mu_1,\mu_2,\beta, a,b\in \R^+$.  Let
 $\{(u_n,v_n) \} \subset D_a\times D_b$ be a bounded sequence such that
\begin{itemize}
\item[(i)] $J(u_n,v_n)\rightarrow C_{(a,b)}$,
\item[(ii)] $ J'(u_n,v_n) + \lambda_{1,n}(u_n,0) + \lambda_{2,n}(0,v_n) \to 0, \quad  \mbox{in } \mathbb{D}^{-1}$ for some real sequences $\{\lambda_{1,n}\}$ and $\{\lambda_{2,n}\}$,
\item[(iii)] $ P(u_n,v_n) \rightarrow 0$,
\item[(iv)] $(u_n,v_n) \rightharpoonup (u,v) $   weakly in  $\mathbb{D}$ and  $(u_n,v_n) \rightarrow (u,v) $ strongly in  $L^{\eta}(\R^2) \times L^{\eta}(\R^2)$ for all $2 < \eta < 2_s$. In addition $u \geq 0$ and $v\geq 0$.
\end{itemize}
Assume that
\begin{equation}\label{strict_inequality}
C_{(a,b)}<\min \{m_{p,\mu_1,a}, m_{q,\mu_2,b}\}.
\end{equation}
where $ m_{p,\mu_1,a}:= J(z_{p,\mu_1,a}, 0) \quad \mbox{and} \quad m_{q,\mu_2,b}:= J(0, z_{q,\mu_2,b})$. Then, up to a subsequence, $(u_n,v_n)\rightarrow (u,v)$ in $\mathbb{D}$ and  $(u,v)\in S_a\times S_b$.
\end{theorem}
\begin{proof}
Without loss of generality, we may assume that $(u_n,v_n)\neq (0,0)$ for all $n\in \mathbb{N}$ since $J(u_n,v_n)\rightarrow C_{a,b}>0=J(0,0)$, see Corollary \ref{corollary4.7}
First note that (ii) can be rewrite as : there exists sequences $\{\lambda_{1,n}\}$ and $\{\lambda_{2,n}\}$ such that
\begin{equation*}
\begin{cases}
 -\partial_{xx}u_n + (-\Delta)_{y}^{s} u_n +\lambda_{1,n}u_n-\mu_1 u_{n}^{p-1}-\beta r_1 u_{n}^{r_1-1}v_{n}^{r_2}=o(1),\\
 -\partial_{xx}v_n + (-\Delta)_{y}^{s} v_n +\lambda_{2,n}v_n-\mu_2 v_{n}^{q-1}-\beta r_2 u_{n}^{r_1}v_{n}^{r_2-1}=o(1).
\end{cases}
\end{equation*}
From the boundedness of $\{(u_n,v_n)\} \subset \mathbb{D}$,  it follows that $\{\lambda_{1,n}\}$ and $\{\lambda_{2,n}\}$ are also bounded and we can assume that $\lambda_{1,n}\rightarrow \lambda_1$ and $\lambda_{2,n}\rightarrow \lambda_2$ for some $\lambda_1, \lambda_2$.
We claim that
\begin{equation}\label{eq:20220830-we1}
u\neq 0, v\neq 0.
\end{equation}
If not, without loss of generality, we may assume that $u=0$. Then we have that $v\neq 0$. If not, from $P(u_n,v_n) \to 0$ we get that
\begin{align*}
&\int_{\mathbb{R}^{2}} \big(|\partial_{x}u_n|^2 +  |(-\Delta)_{y}^{s /2} u_n|^2 \big)\,\mathrm{d}x \mathrm{d}y +\int_{\mathbb{R}^{2}} \big(|\partial_{x}v_n|^2 +  |(-\Delta)_{y}^{s /2} v_n|^2  \big)\,\mathrm{d}x \mathrm{d}y\nonumber\\
&=\frac{(p-2)(1+s)}{2ps}\mu_1\|u_n\|_p^p
+\frac{(q-2)(1+s)}{2qs}\mu_2\|v_n\|_q^q
+\frac{(r_1+r_2-1)(1+s)}{2s}\beta \int_{\R^2} |u_n|^{r_1} |v_n|^{r_2} \mathrm{d}x\mathrm{d}y  + o(1)=o(1),
\end{align*}
where we have used that
$$
\|u_n\|_p^p=o(1), \|v_n\|_q^q=o(1) \mbox{ and } \int_{\R^2}u_{n}^{r_1}v_{n}^{r_2}\mathrm{d}x=o(1),
$$
by the H\"older inequality and Gagliardo-Nirenberg inequality. This contradicts Lemma \ref{lemma4.6} which proves the claim. Now
we know that
$$0<\|v\|_2^2\leq \liminf_{n\rightarrow \infty}\|v_n\|_2^2$$
which implies $v$ is a non-trivial (non-negative) solution to
\begin{equation}\label{eq:20210623-wze1}
-\partial_{xx}v + (-\Delta)_{y}^{s} v +\lambda_{2}v=\mu_2 v^{q-1}\hbox{in}~\R^2, \, v\in \mathcal{H}^{1,s}(\R^2).
\end{equation}
Recalling the Pohozaev identity
\begin{equation*}
\int_{\mathbb{R}^{2}} \big(|\partial_{x}v|^2 +  |(-\Delta)_{y}^{s /2} v|^2\big)\mathrm{d}x\mathrm{d}y=\frac{(q-2)(1+s)}{2qs}\mu_2\|v\|_q^q,
\end{equation*}
and \eqref{eq:20210623-wze1}, we deduce that
$$\lambda_2\|v\|_2^2=\frac{2qs-(q-2){(1+s)}}{2qs}\mu_2 \|v\|_q^q>0.$$
Hence, $\lambda_2>0$. So by
\begin{align*}
&\int_{\mathbb{R}^{2}} \big(|\partial_{x}v_n|^2 +  |(-\Delta)_{y}^{s /2} v_n|^2  \big)\,\mathrm{d}x \mathrm{d}y+\lambda_{2,n}\|v_n\|_2^2=\mu_2\|v_n\|_q^q+\beta r_2\int_{\R^2}u_{n}^{r_1} v_{n}^{r_2}\mathrm{d}x\mathrm{d}y+o(1)=\mu_2\|v\|_q^q+o(1)\\
&=\int_{\mathbb{R}^{2}} \big(|\partial_{x}v|^2 +  |(-\Delta)_{y}^{s /2} v|^2  \big)\,\mathrm{d}x \mathrm{d}y+\lambda_{2}\|v\|_2^2+o(1),
\end{align*}
we have that $v_n\rightarrow v$ in $\mathcal{H}^{1,s}(\R^2)$.
Now using that $P(u_n,v_n) \to 0,$ and $ u_n\rightarrow 0$ in $L^\eta(\R^2)$ for all $2<\eta<2_s$, it is easy to prove that
{$\int_{\mathbb{R}^{2}} \big(|\partial_{x}u|^2 +  |(-\Delta)_{y}^{s /2} u|^2  \big)\,\mathrm{d}x \mathrm{d}y\rightarrow 0$.} So
$$ J(0,v)=\lim_{n\rightarrow \infty}J(u_n,v_n)=C_{(a,b)}.$$
Put $\delta:=\|v\|_2^2$, then $\delta\in (0,b]$.
By Lemma \ref{lemma3.1}, one has that
$v=z_{q,\mu_2,\delta}$ and $C_{(a,b)}=J(0,v)=m_{q,\mu_2,\delta}$.
However, by Lemma \ref{lemma3.2}-(i), and \eqref{strict_inequality} we have
$m_{q,\mu_2,\delta}\geq  m_{q,\mu_2,b}> C_{(a,b)},$
 a contradiction. This ends the proof of \eqref{eq:20220830-we1} and thus
\begin{equation*}
0<\|u\|_2^2\leq \liminf_{n\rightarrow \infty}\|u_n\|_2^2 \quad \mbox{and} \quad 0<\|v\|_2^2\leq \liminf_{n\rightarrow \infty}\|v_n\|_2^2.
\end{equation*}
Now, from the convergence properties (iv), we deduce that
 $(u,v)$ is a non-trivial, solution to \eqref{4.1}. In particular it holds that $P(u,v)=0$ and  by
\begin{align*}
&\int_{\mathbb{R}^{2}} \big(|\partial_{x}u|^2 +  |(-\Delta)_{y}^{s /2} u|^2 \big)\,\mathrm{d}x \mathrm{d}y +\int_{\mathbb{R}^{2}} \big(|\partial_{x}v|^2 +  |(-\Delta)_{y}^{s /2} v|^2  \big)\,\mathrm{d}x \mathrm{d}y \\
&\leq o(1)+\int_{\mathbb{R}^{2}} \big(|\partial_{x}u_n|^2 +  |(-\Delta)_{y}^{s /2} u_n|^2 \big)\,\mathrm{d}x \mathrm{d}y +\int_{\mathbb{R}^{2}} \big(|\partial_{x}v_n|^2 +  |(-\Delta)_{y}^{s /2} v_n|^2  \big)\,\mathrm{d}x \mathrm{d}y\\
=&o(1)+\frac{(p-2)(1+s)}{2ps}\mu_1\|u_n\|_p^p+\frac{(q-2)(1+s)}{2qs}\mu_2\|v_n\|_q^q+\frac{(r_1+r_2-2)(1+s)}{2s}\beta \int_{\R^2} |u_n|^{r_1} |v_n|^{r_2} \mathrm{d}x\mathrm{d}y \\
=&o(1)+\frac{(p-2)(1+s)}{2ps}\mu_1\|u\|_p^p+\frac{(q-2)(1+s)}{2qs}\mu_2\|v\|_q^q+\frac{(r_1+r_2-2)(1+s)}{2s}\beta \int_{\R^2} |u|^{r_1} |v|^{r_2} \mathrm{d}x\mathrm{d}y \\
=&o(1)+\int_{\mathbb{R}^{2}} \big(|\partial_{x}u|^2 +  |(-\Delta)_{y}^{s /2} u|^2 \big)\,\mathrm{d}x \mathrm{d}y +\int_{\mathbb{R}^{2}} \big(|\partial_{x}v|^2 +  |(-\Delta)_{y}^{s /2} v|^2  \big)\,\mathrm{d}x \mathrm{d}y 
\end{align*}
we obtain that
$${\int_{\mathbb{R}^{2}} \big(|\partial_{x}u_n|^2 +  |(-\Delta)_{y}^{s /2} u_n|^2  \big)\,\mathrm{d}x \mathrm{d}y\rightarrow \int_{\mathbb{R}^{2}} \big(|\partial_{x}u|^2 +  |(-\Delta)_{y}^{s /2} u|^2  \big)\,\mathrm{d}x \mathrm{d}y, }$$\
{$$\int_{\mathbb{R}^{2}} \big(|\partial_{x}v_n|^2 +  |(-\Delta)_{y}^{s /2} v_n|^2  \big)\,\mathrm{d}x \mathrm{d}y\rightarrow \int_{\mathbb{R}^{2}} \big(|\partial_{x}v|^2 +  |(-\Delta)_{y}^{s /2} v|^2  \big)\,\mathrm{d}x \mathrm{d}y.$$} Hence, $\displaystyle J(u,v)=\lim_{n\rightarrow \infty}J(u_n,v_n)=C_{(a,b)}.$
\noindent
Recalling \eqref{4.1}, we conclude that $\lambda_1>0,\lambda_2>0$.
Then by
\begin{align*}
&\int_{\mathbb{R}^{2}} \big(|\partial_{x}u_n|^2 +  |(-\Delta)_{y}^{s /2} u_n|^2  \big)\,\mathrm{d}x \mathrm{d}y+\lambda_{1,n}\|u_n\|_2^2=\mu_1\|u_n\|_p^p+\beta r_2\int_{\R^2}u_{n}^{r_1} v_{n}^{r_1}\mathrm{d}x\mathrm{d}y+o(1)\\
=&\mu_1\|u\|_p^p+\beta r_2\int_{\R^2}u_{n}^{r_1} v_{n}^{r_1}\mathrm{d}x\mathrm{d}y+o(1)=\int_{\mathbb{R}^{2}} \big(|\partial_{x}u|^2 +  |(-\Delta)_{y}^{s /2} u|^2  \big)\,\mathrm{d}x \mathrm{d}y+\lambda_{1}\|u\|_2^2+o(1),
\end{align*}
we obtain that $u_n\rightarrow u$ in $\mathcal{H}^{1,s}(\R^2)$. Similarly, we can prove that $v_n\rightarrow v$ in $\mathcal{H}^{1,s}(\R^2)$.  To obtain that $(u,v)\in S_a\times S_b$ and complete the proof we shall use Lemma \ref{lemma8.2} below.
\end{proof}
\begin{lemma}\label{lemma8.2}
Let $(u,v), \lambda_1,\lambda_2$ be as given by Theorem \ref{theorem8.1}. Then $\lambda_1>0$ implies $u\in S_a$ and  $\lambda_2>0$ implies that $v\in S_b$.
\end{lemma}
\begin{proof}
Suppose that $\lambda_1>0$, we shall prove that $u\in S_a$. If not, $\delta:=\|u\|_2^2\in (0, a)$. Then for $l>0$ small enough, we still have that $((1+l)u, v)\in (D_a\times D_b)\backslash\{(0,0)\}$. For any given $l$, by Corollary \ref{corollary4.3}, there exists a unique $t=t(l)>0$ such that
$((1+l)u_t, v_t)\in \mathcal{P}_{(a,b)}$.
Precisely, $t=t(l)$ is determined by
{\allowdisplaybreaks
\begin{align}
(1+l)^2\int_{\mathbb{R}^{2}}& \big(|\partial_{x}u|^2 +  |(-\Delta)_{y}^{s /2} u|^2 \big)\,\mathrm{d}x \mathrm{d}y +\int_{\mathbb{R}^{2}} \big(|\partial_{x}v|^2 +  |(-\Delta)_{y}^{s /2} v|^2  \big)\,\mathrm{d}x \mathrm{d}y\nonumber\\
&=\frac{(p-2)(1+s)}{2ps}\mu_1\|u\|_p^p \, t^{\frac{(p-2)(1+s)-4s}{2}}(1+l)^p+\frac{(q-2)(1+s)}{2qs}\mu_2\|v\|_q^q \, t^{\frac{(q-2)(1+s)-4s}{2}}\nonumber\\
&\quad+\frac{(r_1+r_2-2)(1+s)}{2s}\beta \Big(\int_{\R^2} |u|^{r_1} |v|^{r_2} \mathrm{d}x\mathrm{d}y \Big) t^{\frac{(r_1+r_2-2)(1+s)-4s}{2}}(1+l)^{r_1}.\nonumber
\end{align}}
From the Implicit Function Theorem, we have that $t(l)\in C^1$. Then, since
\begin{align}
&J((1+l)u_t,  v_t)=\frac{1}{2}[(1+l)^2t^{2s}\int_{\mathbb{R}^{2}} \big(|\partial_{x}u|^2 +  |(-\Delta)_{y}^{s /2} u|^2 \big)\,\mathrm{d}x \mathrm{d}y +t^{2s}\int_{\mathbb{R}^{2}} \big(|\partial_{x}v|^2 +  |(-\Delta)_{y}^{s /2} v|^2  \big)\,\mathrm{d}x \mathrm{d}y]\nonumber\\
&\quad-\frac{\mu_1}{p}\|u\|_p^p (1+l)^p \, t^{\frac{(p-2)(1+s)}{2}}
-\frac{\mu_2}{q}\|v\|_q^q \, t^{\frac{(q-2)(1+s)}{2}}-{\beta} \Big(\int_{\R^2} |u|^{r_1} |v|^{r_2} \mathrm{d}x \mathrm{d}y\Big) t^{\frac{(r_1+r_2-2)(1+s)}{2}}(1+l)^{r_1},\nonumber
\end{align}
we have
{\allowdisplaybreaks
\begin{align*}
\frac{d}{dl}J( (1+l)u_t, & v_t)
=(1+l)t^{2s}\int_{\mathbb{R}^{2}} \big(|\partial_{x}u|^2 +  |(-\Delta)_{y}^{s /2} u|^2 \big)\,\mathrm{d}x \mathrm{d}y-\mu_1\|u\|_p^p (1+l)^{p-1} \,  t^{\frac{(p-2)(1+s)}{2}}\nonumber\\
&-{\beta}r_1\int_{\R^2}|u|^{r_1}|v|^{r_2}\mathrm{d}x \mathrm{d}y (1+l)^{r_1-1} \,  t^{\frac{(r_1+r_2-2)(1+s)}{2}}\nonumber\\
&+\left(s(1+l)^2t^{2s-1}\int_{\mathbb{R}^{2}} \big(|\partial_{x}u|^2 +  |(-\Delta)_{y}^{s /2} u|^2 \big)\,\mathrm{d}x \mathrm{d}y+st^{2s-1}\int_{\mathbb{R}^{2}} \big(|\partial_{x}v|^2 +  |(-\Delta)_{y}^{s /2} v|^2  \big)\,\mathrm{d}x \mathrm{d}y\right)t'\nonumber\\
&-\frac{(p-2)(1+s)}{2p}\mu_1\|u\|_p^p (1+l)^p \,  t^{\frac{(p-2)(1+s)-2}{2}} \,t'-\frac{(q-2)(1+s)}{2q}\mu_2\|v\|_q^q t^{\frac{(q-2)(1+s)-2}{2}} \, t'\nonumber\\
&-\frac{(r_1+r_2-2)(1+s)}{2}\beta \Big(\int_{\R^2}|u|^{r_1}|v|^{r_2}\mathrm{d}x \mathrm{d}y \Big) (1+l)^{r_1}\, t^{\frac{(r_1+r_2-2)(1+s)-2}{2}} \, t'.\nonumber\\
\end{align*}
}
here $t=t(l), t'=t'(l)$.
Putting $l=0$ and noting that $t(0)=1$, we have that
\begin{align*}
\frac{d}{dl}J((1+l)u_{t(l)},  v_{t(l)})\Big|_{l=0}
=&\Big[\int_{\mathbb{R}^{2}} \big(|\partial_{x}u|^2 +  |(-\Delta)_{y}^{s /2} u|^2 \big)\,\mathrm{d}x \mathrm{d}y-\mu_1\|u\|_p^p-\beta r_1\int_{\R^2}|u|^{r_1}|v|^{r_2}\mathrm{d}x\mathrm{d}y \Big]\\
&+P(u,v)t'(0)-\lambda_1\|u\|_2^2.
\end{align*}
Then
$$C_{(a,b)}\leq J((1+l)u_{t(l)},  v_{t(l)})<J(u,v)=C_{(a,b)}, \hbox{ for $l>0$ small enough},$$
a contradiction.
Similarly, using that $\lambda_2>0$, we can prove that $v\in S_b$.
\end{proof}
{\bf Proof of Theorem \ref{th:main-t1}:}
Under the assumptions, by Lemma \ref{lemma7.2}, we have that
$
    C_{(a,b)} < \min \{m_{p,\mu_1,a}, m_{q,\mu_2,b}\}.
$
Consider the Palais-Smale sequence $\{(u_n,v_n)\} \subset \mathbb{D}$, which may be found in Section \ref{sec:Palais_Smale}. By Theorem \ref{theorem8.1}, we know that $(u_n,v_n) \to (u,v) \in S_a \times S_b$; specifically, $(u,v) \in \mathcal{P}_{a,b}$ and $J(u,v) = C_{(a,b)}$. Using Remark \ref{remark5.2}, we infer that $J$, restricted to $\mathcal{P}_{(a,b)}$, admits a minimum composed of Steiner symmetric functions, denoted as $(\overline{u}, \overline{v})$. Lemmas \ref{lemma4.8} and \ref{lemma4.9} demonstrate that this minimum is a critical point of $J$ restricted to $D_a \times D_b$. Consequently, the related Lagrange multipliers $\overline{\lambda}_1$ and $\overline{\lambda}_2$ are strictly positive, as proven in Theorem \ref{theorem8.1}. Thus, we have shown that
$$
    (\overline{\lambda}_1,\overline{\lambda}_2,\overline{u},\overline{v}) \in \mathbb{R}^2 \times {\mathbb{D}}
$$
is a ground state of Problem \eqref{eq:20220902-maine1} with the desired symmetry properties.
\hfill$\Box$

\section*{Acknowledgment}

A.D. is supported by DST INSPIRE Fellowship with sanction number DST/INSPIRE Fellowship/2022/IF220580. A.E. is supported by Nazarbayev University under the Faculty Development Competitive Research Grants Program for 2023-2025 (grant number 20122022FD4121). T.M. is supported by CSIR-HRDG grant with grant sanction No. 25/0324/23/EMR-II.

\section*{Competing Interest statement}
No potential conflict of interest was reported by the authors.

\bibliography{ref}
\bibliographystyle{siam}

\end{document}